\documentclass[reqno]{amsart}
\usepackage{amssymb}
\usepackage{amsmath}

\makeatletter
\@addtoreset{equation}{section}
\makeatother

\renewcommand\thefigure{\thesection.\@arabic\c@figure}
\renewcommand\thetable{\thesection.\@arabic\c@table}

\newtheorem{theorem}{Theorem}[section]

\newtheorem{lemma}[theorem]{Lemma}
\newtheorem{proposition}[theorem]{Proposition}

\newcommand{\mc}[1]{{\mathcal #1}}
\newcommand{\mf}[1]{{\mathfrak #1}}
\newcommand{\mb}[1]{{\mathbf #1}}
\newcommand{\bb}[1]{{\mathbb #1}}

\newcommand{\<}{\langle}
\renewcommand{\>}{\rangle}

\def\R{\mathbb R}
\def\N{\mathbb N}
\def\V{\mathcal V}

\def\N{\mathbb N}
\def\M{\mathcal M}
\def\T{\bb T}
\def\d{\mf d}

\begin{document}

\author{Milton Jara, Claudio Landim, and Sunder Sethuraman}

\address{\noindent IMPA, Estrada Dona Castorina 110,
CEP 22460 Rio de Janeiro, Brasil
\newline
e-mail:  \rm \texttt{mjara@impa.br}
}

\address{\noindent IMPA, Estrada Dona Castorina 110,
CEP 22460, Rio de Janeiro, Brasil and CNRS UMR 6085,
Avenue de l'Universit\'e, BP.12, Technop\^ole du Madrillet,
F76801 Saint-\'Etienne-du-Rouvray, France.
\newline
e-mail:  \rm \texttt{landim@impa.br}
}

\address{\noindent Department of Mathematics, Iowa State University,
  Ames, IA  50011
\newline
e-mail:  \rm \texttt{sethuram@iastate.edu}
}

\title[Nonequilibrium fluctuations in sublinear zero-range processes]
{Nonequilibrium fluctuations for a tagged particle in one-dimensional
  sublinear rate zero-range processes}

\begin{abstract}
  Nonequilibrium fluctuations of a tagged, or distinguished particle
  in a class of one dimensional mean-zero zero-range systems with
  sublinear, increasing rates are derived.  In Jara-Landim-Sethuraman
  (2009), processes with at least linear rates are considered.

  A
  different approach to establish a main ``local replacement'' limit
  is required for sublinear rate systems, given that their mixing
  properties are much different.
  The method discussed also allows to capture the fluctuations of a
  ``second-class'' particle in unit rate, symmetric zero-range models.
  \end{abstract}



\thanks{S. Sethuraman was partially supported by NSF 0906713 and NSA Q
  H982301010180 }

\maketitle

\section{Introduction and Results}
\label{sec0}

Zero-range processes follow a collection of random walks on a lattice
which interact in the following way: Informally, a particle at a
location with $k$ particles displaces by $j$ with infinitesimal rate
$(g(k)/k)p(j)$ where the process rate $g: \bb N_0\to \bb R_+$ is a
function on the non-negative integers, and $p(\cdot)$ is a
translation-invariant single particle transition probability.  These
processes have served as formal models for types of queuing, traffic,
fluid, granular flow etc.  A review of some of the applications can be
found in \cite{Evans}.

Different behaviors may be found by varying the choice of rate $g$,
when say $p$ is symmetric and nearest-neighbor.  For instance, the
spectral gap or mixing properties of the system defined on a cube of
width $n$ with $k$ particles depend strongly on the asymptotic growth
of $g$.  For a class of models, when $g$ is on linear order, the gap
is order $n^{-2}$ and does not depend on $k$ \cite{LSV}.  However,
when $g$ is the unit rate, $g(x)=\mb 1\{x\geq 1\}$, the gap is of
order $n^{-2}(1+\rho)^{-2}$ where $\rho = k/n$ \cite{Morris}.  Also,
when $g$ is sub-linear, of form $g(x)=x^\gamma$ for $0<\gamma\le 1$,
the gap is of order $n^{-2}(1+\rho)^{\gamma-1}$ \cite{Nagahata}.

We will consider ``attractive'' models, that is those with increasing
rates $g$, on one dimensional tori $\T_N = \bb Z/N\bb Z$.  We will
also assume $g$ is either bounded or sublinear of a certain type.  In
addition, we suppose the jump probability $p$ is finite-range and
mean-zero.  The aim of the article is to understand certain
``nonequilibrium'' scaling limits of a distinguished, or tagged
particle in this setting.

Because of the particle interaction, the tagged particle is not
Markovian with respect to its own history.  However, one expects that
its position to homogenize to a diffusion with parameters given in
terms of the ``bulk'' hydrodynamic density.  Although fluctuations of
Markov processes are much examined (cf. Komorowski-Landim-Olla
\cite{Ko_La_Ol}), and there are many central limit theorems for types
of tagged particles when the system is in ``equilibrium'' (cf.
Kipnis-Varadhan \cite{KV}, Saada \cite{Saada}, Sethuraman
\cite{Szrtg}), much less is understood when particles {\it both}
interact nontrivially, and begin in ``nonequilibrium''.

In this context, the only previous work treating a general class of
interacting particle systems in a systematic way is
Jara-Landim-Sethuraman \cite{jls} which proves a nonequilibrium
functional central limit theorem for a class of zero-range processes
whose rates $g$ have at least linear growth, that is $g(k)\geq c_1k$
for a $c_1>0$.  The proof in \cite{jls} relies on an important
estimate, a ``local'' hydrodynamic limit, which however makes strong
use of the linear growth of $g$, in particular that the spectral gap
on a localized cube does not depend on the number of particles in the
cube.  Unfortunately, this proof does not carry over to the bounded or
sublinear rate case.

A main contribution of this article is to supply a different approach
for the main ``local replacement'' (Theorem \ref{replacement1}) with
respect to a class of increasing, bounded or sublinear rate zero-range
models so that the nonequilibrium limit for the tagged particle can be
established (Theorem \ref{th2}).  As in \cite{jls}, a consequence of
the argument is that the limit of the empirical density in the
reference frame of the tagged particle can be identified as the
hydrodynamic density in the frame of the limit tagged particle
diffusion (Theorem \ref{th1}).

We remark the approach taken here with respect to the ``local
replacement'' is robust enough so that it can apply to determine the
nonequilibrium fluctuations of a ``second-class'' particle, and
associated reference frame empirical density, in the symmetric unit
rate case, that is when $g(k)=\mb 1\{k\geq 1\}$ (Theorems
\ref{th2_second}, \ref{th1_second}).  This is the first work to
address a nonequilibrium central limit theorem for a second-class
particle.

Finally, we mention other central limit theorems for a tagged particle
which take advantage of special features in types of exclusion and
interacting Brownian motion models can be found in Jara and Landim
\cite{JL}, Jara \cite{Jara},and Grigorescu \cite{Grig}.  Note also
``propagation of chaos'' results yield homogenization limits for the
averaged tagged particle position in simple exclusion, Rezakhanlou
\cite{Reza-pr}.  \medskip

Let now $\xi_t= \{\xi_t(x): x\in \bb T_N\}$ be the zero-range process on
$\bb T_N=\bb Z / N \bb Z$ with single particle transition probability
$p(\cdot)$ and process rate $g:\bb N_0\to \bb R_+$.  We will assume
that $g(0)=0$, $g(1)>0$, and that $g$ is increasing (or
``attractive''), $g(k+1)\geq g(k)$ for $k\geq 1$.  In addition,
throughout the paper, and in all results, we impose one of the
following set of conditions (B) or (SL):
\begin{itemize}
\item[(B)] $g$ is bounded: For $k\geq 1$, there are constants
  $0<a_0\leq a_1$ such that $a_0 \leq g(k) \leq a_1$.
\end{itemize}
Before specifying the class of sublinear rates considered, let $W(l,k)$ be the
inverse of the spectral gap of the process defined on the cube
$\Lambda_l = \{-l,\ldots, l\}$ with $k$ particles, when the transition probability $p$ is symmetric and nearest-neighbor  (cf. Section
\ref{spec_gap_section} for more definitions).

\begin{itemize}
\item[(SL1)] $g$ is sublinear: $\lim_{k\to\infty} g(k)=\infty$,
  $g(k)/k: \N \rightarrow \R_+$ is decreasing, and $\lim_{k\rightarrow
    \infty} g(k)/k = 0$.  In particular, since $g$ is increasing,
  there exists constants $a_0, a_1>0$ such that $a_0\leq g(k)$ and
  $g(k)/k\leq a_1$, $k\geq 1$.

\item[(SL2)] $g$ is \emph{Lipschitz}: There is a constant $a_2$ such
  that $|g(k+1)-g(k)|\leq a_2$ for $k\geq 0$.

\item[(SL3)] The spectral gap satisfies, for all constants $C$ and
  $l\geq 1$, that
\begin{equation}
\lim_{N\uparrow\infty} N^{-1}\max_{1\leq k\leq C\log N} k^2W(l,k) \ =\ 0.
\label{spec_gap_condition}
\end{equation}
\end{itemize}

It is proved in Lemma \ref{spec_gap} that all processes with bounded
rates $g$ satisfy (\ref{spec_gap_condition}).  In addition, by the
spectral gap estimate \cite{Nagahata}, processes with rates $g(k) =
k^{\gamma}$ for $0<\gamma\leq 1$ satisfy (\ref{spec_gap_condition}).
In addition, we will assume that $p$ is finite-range, irreducible, and
mean-zero, that is

\begin{itemize}
\item[(MZ)] There exists $R>0$ such that $p(z)=0$ for $|z|>R$, and
  $\sum zp(z)=0$.
\end{itemize}

We also will take the scaling parameter $N$ larger than the support of
$p(\cdot)$.

Denote by $\Omega_N = \bb N_0^{\bb T_N}$ the state space and by
$\xi$ the configurations of $\Omega_N$ so that $\xi(x)$, $x\in \bb
T_N$, stands for the number of particles in site $x$ for the
configuration $\xi$.
The zero-range process is a continuous-time Markov
chain generated by
\begin{equation}
\label{c0} (\mc L_N f) (\xi) \;=\; \sum_{x\in \bb T_N } \sum_{z}
p(z) \, g(\xi(x))\, \big[f(\xi^{x,x+z}) -f(\xi)\big]\; ,
\end{equation}
where $\xi^{x,y}$ represents the configuration obtained from $\xi$ by
displacing a particle from $x$ to $y$:
\begin{equation*}
\xi^{x,y}(z) =
\begin{cases}
\xi(x)-1 & {\rm for \ } z=x \\
\xi(y)+1 &{\rm for \ } z=y \\
\xi(z) &{\rm for \ } z \neq x,y\;.
\end{cases}
\end{equation*}

The zero-range process $\xi(t)$ has a well known explicit family
product invariant measures $\bar\mu_\varphi$, $0\leq\varphi<\lim g(k)
=: g(\infty)$, on $\Omega_N$ defined on the nonnegative integers,
$$
\bar \mu_\varphi(\xi(x)=k) \ = \ \frac{1}{Z_\varphi}\frac{\varphi^k}{g(k)!}
 {\rm \ \ for \ \ } k\geq 1 \ \ \ {\rm and \ \ }
\bar\mu_\varphi(\xi(x)=0) \ = \ \frac{1}{Z_\varphi}
$$
where $g(k)!=g(1)\cdots g(k)$ and $Z_\varphi$ is the normalization.
Denote by $\rho(\varphi)$ the mean of the marginal $\bar\mu_\varphi$,
$\rho(\varphi) = \sum_k k\mu_\varphi(\xi(x)=k)$.  Since $g$ is
increasing, the radius of convergence of $Z_\varphi$ is $g(\infty)$,
and $\lim_{\varphi \uparrow g(\infty)} \rho(\varphi)=\infty$. As
$\rho(0)=0$ and $\rho(\varphi)$ is strictly increasing, for a given
$0\leq \rho<\infty$, there is a unique inverse $\varphi =
\varphi(\rho)$.  Define then the family in terms of the density $\rho$
as $\mu_{\rho}= \bar \mu_{\varphi(\rho)}$.

Now consider an initial configuration $\xi$ such that $\xi(0) \geq 1$,
and let $\Omega^*\subset \Omega$ be the set of such configurations.
Distinguish, or tag one of the particles initially at the origin, and
follow its trajectory $X_t$, jointly with the evolution of the process
$\xi_t$.  It will be convenient for our purposes to consider the
process as seen by the tagged particle.  This reference process
$\eta_t(x) = \xi_t(x + X_t)$ is also Markovian and has generator in
form $L_N = L_N^{env} + L_N^{tp}$, where $L_N^{env}$, $L_N^{tp}$ are
defined by
\begin{equation}
\label{L^{env}}
\begin{split}
(L_N^{env} f) (\eta) &= \sum_{x \in \bb T_N\setminus\{ 0\}}
\sum_{z} p(z) \,
g(\eta(x)) \, [f(\eta^{x,x+z})-f(\eta)] \\
&+  \sum_{y} p(y) \, g(\eta(0)) \, \frac{\eta(0) -1}{\eta(0)} \,
[f(\eta^{0,y})-f(\eta)]\;, \\
(L_N^{tp} f) (\eta) &= \sum_{z} p(z) \,
\frac{g(\eta(0))}{\eta(0)} \, [f(\theta_z \eta)-f(\eta)]\;.
\end{split}
\end{equation}
In this formula, the translation $\theta_z$ is defined by
\begin{equation*}
(\theta_z \eta)(x) =
\begin{cases}
\eta(x+z) & {\rm for \ } x \neq 0,-z \\
\eta(z)+1 &{\rm for \ } x=0 \\
\eta(0)-1 &{\rm for \ } x =-z.\\
\end{cases}
\end{equation*}
The operator $L_N^{tp}$ corresponds to jumps of the tagged particle,
while the operator $L_N^{env}$ corresponds to jumps of the other
particles, called environment.

A key feature of the tagged motion is that it can be written as a
martingale in terms of the reference process:
\begin{equation}
\label{decomposition}
X_t \;=\; \sum_j j\, N_t(j)
\;=\; \sum_j jM_t(j) \;+\; \mf m
\int_0^t \frac{g(\eta_s(0))}{\eta_s(0)}ds \ = \ \sum_j j\, M_t(j)\;,
\end{equation}
where $\mf m = \sum_j j p(j) = 0$ is the mean drift, $N_t(j)$ is the
counting process of translations of size $j$ up to time $t$, and
$M_t(j) = N_t(j)-p(j) \int_0^t g(\eta_s(0))/\eta_s(0) ds$ is its
corresponding martingale.  In addition, $M^2_t(j) -
p(j)\int_0^tg(\eta_s(0))/\eta_s(0)ds$ are martingales which are
orthogonal as jumps are not simultaneous a.s.  Hence, the quadratic
variation of $X_t$ is $\<X\>_t = \sigma^2\int_0^t
{g(\eta_s(0))}/{\eta_s(0)}ds$ where $\sigma^2 = \sum j^2 p(j)$.

For the reference process $\eta_t$, the ``Palm'' or origin size biased
measures given by $d\nu_\rho = (\eta(0)/\rho)d\mu_\rho$ are invariant
(cf. \cite{Port}, \cite{Saada}).  Note that $\nu_\rho$ is also a
product measure whose marginal at the origin differs from that at
other points $x\neq 0$.  Here, we take $\nu_0 = \delta_{\d_0}$, the
Dirac measure concentrated on the configuration $\d_0$ with exactly
one particle at the origin, and note that $\nu_\rho$ converges to
$\delta_{\mf d_0}$ as $\rho\downarrow 0$.

The families $\{\mu_\rho: \rho \ge 0\}$ and $\{\nu_\rho: \rho\ge 0\}$
are stochastically ordered.  Indeed, this follows as the marginals of
$\mu_\rho$ and $\nu_\rho$ are stochastically ordered.  Also, since we
assume that $g$ is increasing, the system is ``attractive,'' that is
by the ``basic coupling'' (cf. Liggett \cite{Liggett}) if $dR$ and
$dR'$ are initial measures of two processes $\xi_t$ and $\xi'_t$, and
$dR\ll dR'$ in stochastic order, then the distributions of $\xi_t$ and
$\xi'_t$ are similarly stochastically ordered \cite{Liggett}.  We also
note, when $p$ is symmetric, that $\mu_\rho$ and $\nu_\rho$ are
reversible with respect to $\mathcal L_N$, and $L_N$ and $L^{env}_N$
respectively.  \medskip

From this point, to avoid uninteresting compactness issues, we define
every process in a finite time interval $[0,T]$, where $T<\infty$ is
fixed. Let $\bb T$ be the unit torus and let $\mc M_+(\bb T)$ be the
set of positive Radon measures in $\bb T$.

For a continuous, positive function $ \rho_0: \bb T \to \bb R_+$, define
$\mu^N=\mu^N_{\rho_0(\cdot)}$ as the product measure in $\Omega_N$ given
by $\mu^N_{\rho_0(\cdot)}(\eta(x)=k)=\mu_{\rho_0(x/N)}(\eta(x)=k)$.

Consider the process $\xi_t^N =: \xi_{tN^2}$, generated by $N^2 \mc
L_N$ starting from initial measure $\mu^N$. Define the process
$\pi_t^{N,0}$ in $\mc D([0,T],\mc M_+(\bb T))$ as
\begin{equation*}
\pi_t^{N,0}(du) = \frac{1}{N} \sum_{x \in \bb T_N} \xi_t^N (x)
\delta_{x/N} (du)\;,
\end{equation*}
where $\delta_u$ is the Dirac distribution at point $u$.

The next result, ``hydrodynamics,'' under the assumption $p(\cdot)$ is
mean-zero, is well known (cf. De Masi-Presutti \cite{dmp},
Kipnis-Landim \cite{kl}).

\begin{theorem}
\label{th0}
For each $0\le t\le T$, $\pi_t^{N,0}$ converges in probability to the
deterministic measure $\rho(t,u)du$, where $\rho(t,u)$ is the solution
of the hydrodynamic equation
\begin{equation}
\label{ec0} \left\{
\begin{array}{l}
\partial_t \rho = \sigma^2 \partial_x^2 \varphi(\rho) \\
\rho(0,u) = \rho_0(u),\\
\end{array}
\right.
\end{equation}
and $\varphi(\rho)= \int g(\xi(0)) d \mu_\rho$.
\end{theorem}

We now state results for the tagged particle motion.  Define the
product measure $\nu^N=\nu_{\rho_0(\cdot)}^N$ in $\Omega^*_N$ given by
$\nu_{\rho_0(\cdot)}^N(\eta(x)=k) = \nu_{\rho_0(x/N)}(\eta(x) =k)$,
and let $\eta_t^N=: \eta_{tN^2}$ be the process generated by $N^2 L_N$
and starting from the initial measure $\nu^N$. Define the empirical
measure $\pi_t^N$ in $\mc D([0,T],\mc M_+(\bb T))$ by
\begin{equation*}
\pi_t^N(du) = \frac{1}{N} \sum_{x\in \bb T_N} \eta_t^N(x)
\delta_{x/N}(du).
\end{equation*}
Let also $X_t^N = X_{N^2t}$ be the position of the tagged particle at
time $N^2t$.

Define also the continuous function $\psi:\bb R_+ \to \bb R_+$ by
$$
\psi(\rho)  = \int \frac{g(\eta(0))}{\eta(0)}\, d\nu_\rho\;.
$$
Note $\psi(\rho) = \varphi(\rho)/\rho$ for $\rho>0$, and $\psi(0)=
g(1)$.  The first main result of the article is to identify the
scaling limit of the tagged particle as a diffusion process:

\begin{theorem}
\label{th2}
Let $x_t^N = X^N_t/N$ be the rescaled position of the
tagged particle for the process $\xi_t^N$. Then, $\{x_t^N : t\in
[0,T] \}$ converges in distribution in the uniform topology to the
diffusion $\{x_t : t\in [0,T]\}$ defined by the stochastic
differential equation
\begin{equation}
\label{c9} d x_t = \sigma \, \sqrt{\psi(\rho(t,x_t))} \, dB_t\; ,
\end{equation}
where $B_t$ is a standard Brownian motion on $\bb T$.
\end{theorem}

In terms of this characterization, we can describe the evolution of the
empirical measure as seen from the tagged particle:

\begin{theorem}
\label{th1}
We have $\{\pi_t^N : t\in [0,T]\}$ converges in distribution with
respect to the Skorohod topology on $\mc D([0,T],\mc M_+(\bb T))$ to
the measure-valued process $\{\rho(t,u +x_t)du : t\in [0,T]\}$, where
$\rho(t,u)$ is the solution of the hydrodynamic equation (\ref{ec0})
and $x_t$ is given by \eqref{c9}.
\end{theorem}

\medskip

When the rate $g(k)=\mb 1\{k\geq 1\}$, scaling limits of a
``second-class'' particle $\mathcal X_t$ can also be captured.
Informally, such a particle must wait until all the other particles,
say ``first-class'' particles, have left its position before it can
displace by $j$ with rate $p(j)$.  More precisely, its dynamics can be
described in terms of its reference frame motion.  For an initial
configuration $\xi$ such that $\xi(0)\geq 1$, let $\zeta_t(x) =
\xi_t(x + \mathcal X_t) - \delta_{0,x}$, where $\delta_{a,b}$ is
Kronecker's delta, be the system of first-class particles in the
reference frame of the second-class particle.  The generator $\mf L_N$
takes form $\mf L_N = \mf L^{env}_N + \mf L_N^{tp}$, where
\begin{equation*}
\begin{split}
& (\mf L^{env}_N f) (\zeta) \;=\; \sum_{x \in \bb T_N} \sum_{z} p(z) \,
 \mb 1\{\zeta(x)\geq 1\} \, [f(\zeta^{x,x+z})-f(\zeta)]\; , \\
& (\mf L_N^{tp} f) (\zeta) \;=\; \sum_{z} p(z) \, \mb 1\{\zeta(0)=0\} \,
 [f(\tau_z \zeta)-f(\zeta)]\;
\end{split}
\end{equation*}
where $\tau_z$ is the pure spatial translation by $z$, $(\tau_z \zeta
)(y) = \zeta (y+z)$ for $y\in \bb T^N$.

Then, as for the regular tagged particle, we have
\begin{equation*}
\mathcal X_t \;=\; \sum_j j\mc N_t(j)
\;=\; \sum_j j\M_t(j) \;+\; \mf m \int_0^t
\mb 1\{\zeta_s(0)=0\}ds \ = \ \sum_j j\M_t(j)
\end{equation*}
where, $\mf m = \sum_j j p(j) = 0$ is the mean drift, $\mc N_t(j)$ is
the counting process of translations of size $j$ up to time $t$, and
$\M_t(j) = \mc N_t(j)-p(j)\int_0^t \mb 1\{\zeta_s(0)=0\}ds$ are the
associated martingales.  As before, since $\M^2_t(j) - p(j)\int_0^t
\mb 1\{\zeta_s(0)=0\}ds$ are orthogonal martingales, the quadratic
variation of $\mathcal X_t$ is $\<\mc X\>_t = \sigma^2\int_0^t \mb
1\{\zeta_s(0)=0\}ds$.

For the second-class reference process $\zeta_t$, under the assumption
$p(\cdot)$ is symmetric, the family $d\kappa_\rho =
(1+\rho)^{-1}(\zeta(0) +1)d\mu_\rho$ for $\rho>0$ are invariant. We
remark that symmetry of $p(\cdot)$ is needed to show $\kappa_\rho$ are
invariant with respect to the second-class tagged process.

Let $\kappa^N = \kappa^N_{\rho(\cdot)}$ be the product measure with
$\kappa^N_{\rho_0(\cdot)}(\zeta(x)=k) = \kappa_{\rho_0(x/N)}
(\zeta(x)=k)$. Let also $\zeta^N_t = \zeta_{N^2t}$ be the process
generated by $N^2\mf L_N$ starting from $\kappa^N$.  Correspondingly,
define empirical measure
$\pi^{N,1}_t(du)=(1/N)\sum_{x\in\T_N}\zeta^N_t(x)\delta_{x/N}(du)$.
In addition, let
$$
\chi(\rho) \ = \ \int \mb 1\{\zeta(0)=0\}d\kappa_\rho \ =\
\frac{1}{1+\rho}\int \mb 1\{\zeta(0)=0\} d\mu_\rho \ = \
\frac{1}{(1+\rho)^2}\;.
$$

In the case $g(k) = \mb 1\{k\ge 1\}$, $\varphi(\rho) =
\rho/[1+\rho]$. Denote by $\rho^1$ the solution of (\ref{ec0}) with
such function $\varphi$.  We may now state results for the
second-class tagged motion.

\begin{theorem}
\label{th2_second}
Suppose $p(\cdot)$ is symmetric.  Let $y_t^N = \mc X^N_t/N$ be the
rescaled position of the second-class tagged particle for the process
$\zeta_t^N$. Then, $\{y_t^N : t\in [0,T] \}$ converges in distribution
in the uniform topology to the diffusion $\{y_t : t\in [0,T]\}$
defined by the stochastic differential equation
\begin{equation}
\label{c9_second} d y_t = \sigma \, \sqrt{\chi(\rho^1(t,y_t))} \, dB_t\; ,
\end{equation}
where $B_t$ is a standard Brownian motion on $\bb T$.
\end{theorem}

\begin{theorem}
\label{th1_second}
Suppose $p(\cdot)$ is symmetric.  Then, $\{\pi_t^{N,1} : t\in [0,T]\}$
converges in distribution with respect to the Skorohod topology on
$\mc D([0,T],\mc M_+(\bb T))$ to the measure-valued process
$\{\rho^1(t,u +y_t)du : t\in [0,T]\}$ where $y_t$ is given by
\eqref{c9_second}.
\end{theorem}

The outline of the proofs of Theorems \ref{th2}, \ref{th1},
\ref{th2_second} and \ref{th1_second} are given at the end of this
section.  We now state the main replacement estimate with respect to
process $\eta^N_t$ for a (regular) tagged particle.  A similar
estimate holds with respect to process $\zeta^N_t$ and a
``second-class'' particle, stated in the proof of Theorem
\ref{th2_second}.  As remarked earlier, this replacement estimate is
the main ingredient to show Theorems \ref{th2} and \ref{th1}.
\medskip

Denote by $\bb P_\nu$ the probability measure in $\mc
D([0,T],\Omega_N)$ induced by the process $\eta_t^N$, starting from
initial measure $\nu$, and by $\bb E_\nu$ the corresponding
expectation.  When $\nu=\nu^N$, we abbreviate $\bb P_{\nu^N} = \bb
P^N$ and $\bb E_{\nu^N}= \bb E^N$.  With respect to process $\xi_t^N$,
denote $\mb P_\mu$ the probability measure in $\mc D([0,T],\Omega_N)$
starting from measure $\mu$, and $\mb E_\mu$ the associated
expectation.  Denote also by $E_\mu[h]$ and $\<h\>_\mu$ the
expectation of a function $h: \Omega_N \to \bb R$ with respect to the
measure $\mu$; when $\mu = \nu_\rho$, let $E_\rho[h]$, $\<h\>_\rho$
stand for $E_{\nu_\rho}[h]$, $\<h\>_{\nu_\rho}$.  Define also inner
product $\<f,g\>_\mu = E_\mu[fg]$, and covariance $\<f;g\>_\mu =
E_\mu[fg] - E_\mu[f]E_\mu[g]$ with the same convention when
$\mu=\nu_\rho$.  To simplify notation, we will drop the superscript
$N$ in the speeded-up process $\eta_t^N$.

For $\varepsilon>0$ and $l\geq 0$, let
$$
\eta^l(x) = \frac{1}{2l+1}\sum_{|y|\leq l} \eta(x+y)\;.
$$

\begin{theorem}
\label{replacement1}
Let $h:\N\rightarrow \R_+$ be a positive, bounded, Lipschitz function
such that there exists a constant $C$ such that $h(k)\leq C[g(k)/k]$
for $k\geq 1$.  Then,
\begin{equation*}
\limsup_{l\rightarrow \infty}\limsup_{\epsilon \to 0}
\limsup_{\varepsilon \to 0} \limsup_{N \to \infty}
\bb E^N \Big[\,\Big|\int_0^t h(\eta_s(0)) -
\frac{1}{\epsilon N}\sum_{x=1}^{\epsilon N}
\bar{H_l}(\eta^{\varepsilon N}_s(x))\, ds \Big|\, \Big] = 0 \;,
\end{equation*}
where $H(\rho) = E_{\nu_\rho}[h(\eta(0))]$, $H_l(\eta) =
H(\eta^l(0))$, and $\bar{H_l}(\rho) = E_{\mu_\rho}[H_l]$.
\end{theorem}

We now give the outlines of the proof of the main theorems.
\medskip

\noindent {\it Proofs of Theorems \ref{th2} and \ref{th1}.}  First,
the replacement estimate, Theorem \ref{replacement1}, applies when
$h(k) = g(k)/k$: Under the assumptions on $g$, clearly $h$ is
positive, bounded, and Lipschitz.  Given Theorem \ref{replacement1},
the proof of the main theorems straightforwardly follow the same steps
as in \cite{jls}.  Namely, (1) tightness is proved for $(x_t^N, A_t^N,
\pi_t^{N,0},\pi_t^N)$ where $A_t^N=\<x_t^N\>$ is the quadratic
variation of the martingale $x_t^N$.  (2) Using the hydrodynamic
limit, Theorem \ref{th0}, one determines the limit points of
$\pi_t^{N,0}$, and $\pi_t^N = \tau_{x_t^N}\pi_t^{N,0}$.  Limit points
of $A_t^N$ are obtained through the replacement estimate, Theorem
\ref{replacement1}.  Finally, one obtains limits of $x_t^N$ are
characterized as continuous martingales with certain quadratic
variations.  Theorems \ref{th2} and \ref{th1} follow now by Levy's
theorem.  More details on these last points can be be found in
\cite{jls}. \qed

\medskip
\noindent {\it Proofs of Theorems \ref{th2_second} and
  \ref{th1_second}.}  The proofs follow the same scheme as for
Theorems \ref{th2} and \ref{th1}, given a replacement estimate.  One
can rewrite Theorem \ref{replacement1} in terms of $\zeta^N_s$:
\begin{eqnarray*}
&&\limsup_{l\rightarrow \infty}\limsup_{\epsilon \to 0}
\limsup_{\varepsilon \to 0} \limsup_{N \to \infty} \\
&&\ \ \ \ \ \ \ \ \ \ \ \ \bb E^N_{sec} \Big[\,
\Big|\int_0^t h(\zeta_s(0)) - \frac{1}{\epsilon N}\sum_{x=1}^{\epsilon
  N} \bar{H_l}(\zeta^{\varepsilon N}_s(x))\, ds \Big|\, \Big] \;=\; 0 \; .
\end{eqnarray*}
\begin{equation*}
\limsup_{l\rightarrow \infty}\limsup_{\epsilon \to 0}
\limsup_{\varepsilon \to 0} \limsup_{N \to \infty}  \bb E^N_{sec} \Big[\,
\Big|\int_0^t h(\zeta_s(0)) - \frac{1}{\epsilon N}\sum_{x=1}^{\epsilon
  N} \bar{H_l}(\zeta^{\varepsilon N}_s(x))\, ds \Big|\, \Big] \;=\; 0 \; .
\end{equation*}
Here, $h(k) = \mb 1\{k=0\}$, $H(\rho) = E_{\kappa_\rho}[h(\zeta(0))]$,
$H_l(\zeta) = H(\zeta^l(0))$, and $\bar{H_l}(\rho) =
E_{\mu_\rho}[H_l]$.  Also, $\bb E^N_{sec}$ is the process expectation
with respect to $\zeta_s^N$.

Given $d\kappa_\rho = (1+\rho)^{-1}d\mu_\rho + \rho(1+\rho)^{-1}
d\nu_\rho$, the proof of this replacement follows quite closely the
proof of Theorem \ref{replacement1} with straightforward
modifications.  \qed

\medskip The plan of the paper now is to give some spectral gap
estimates, ``global,'' ``local 1-block'' and ``local 2-blocks''
estimates in sections \ref{spec_gap_section}, \ref{global_section},
\ref{1_block_section}, and \ref{2_block_section}, which are used to
give the proof of Theorem \ref{replacement1} in Section
\ref{replacement1_section}.

For simplicity in the proofs, we will suppose that $p(\cdot)$ is
symmetric, and nearest-neighbor, but our results hold, with
straightforward modifications, when $p(\cdot)$ is finite-range,
irreducible, and mean-zero, because mean-zero zero-range processes are
gradient processes.

\section{Spectral gap estimates}
\label{spec_gap_section}

We discuss some spectral gap bounds which will be useful in the
sequel.  For $l\geq 0$, let $\Lambda_l = \{x: |x|\leq l\}$ be a cube
of length $2l+1$ around the origin, and let $\nu^{\Lambda_l}_\rho$ and
$\mu^{\Lambda_l}_\rho$ be the measures $\nu_\rho$ and $\mu_\rho$
restricted to $\Lambda_l$.

For $j\geq 0$, define the sets of
configurations
$\Sigma_{\Lambda_l,j} = \{\eta\in \bb N_0^{\Lambda_l}:
\sum_{x\in \Lambda_l}\eta(x)=j\}$,
and $\Sigma^*_{\Lambda_l, j} \;=\;
\{\eta\in \bb N_0^{\Lambda_l} : \eta(0)\ge 1 , \sum_{x\in\Lambda_l}
\eta(x) = j \}$.
  Define also the canonical measures $\nu_{\Lambda_l,
  j}(\cdot) = \nu^{\Lambda_l}_\rho(\,\cdot\,|
\Sigma^*_{\Lambda_l,j})$, and $\mu_{\Lambda_l,j} (\,\cdot\,)=
\mu^{\Lambda_l}_\rho(\,\cdot\,|\Sigma_{\Lambda_l,j})$.  Note that both
$\nu_{\Lambda_l, j}$ and $\mu_{\Lambda_l,j}$ do not depend on $\rho$.

Denote by $\mathcal L_{\Lambda_l}$, $L^{env}_{\Lambda_l}$ the
restrictions of the generators $\mathcal L_N$, $L^{env}_N$ on
$\Sigma_{\Lambda_l,j}$, $\Sigma^*_{\Lambda_l,j}$, respectively.  These
generators are obtained by restricting the sums over $x,y,z$ in
(\ref{c0}) and (\ref{L^{env}}) to $x$, $x+z$, $y\in \Lambda_l$.
Clearly, $\nu_{\Lambda_l, j}$, $\mu_{\Lambda_l,j}$ are invariant with
respect to $L^{env}_{\Lambda_l}$, $\mathcal L_{\Lambda_l}$,
respectively.  Denote the Dirichlet forms $D(\mu_{\Lambda_l,j}, f) =
\<f,(-\mathcal L_{\Lambda_l} f)\>_{\mu_{\Lambda_l,j}}$ and
$D(\nu_{\Lambda_l,j},f) = \<f, (-L^{env}_{\Lambda_l}
f)\>_{\nu_{\Lambda_l.j}}$.  One can compute
\begin{eqnarray*}
D(\mu_{\Lambda_l,j}, f) &=& \frac{1}{2}\sum_{x,y\in \Lambda_l}
p(y-x)E_{\mu_{\Lambda_l,j}}\left[
  g(\eta(x))\big(f(\eta^{x,y})-f(\eta)\big)^2
\right]\\
D(\nu_{\Lambda_l,j},f) &=& \frac{1}{2}
\sum_{x\in \Lambda_l\setminus\{0\}}\sum_{y\in \Lambda_l}p(y-x)
E_{\nu_{\Lambda_l,j}}\left[ g(\eta(x))
\big(f(\eta^{x,y})-f(\eta)\big)^2\right]\\
&&\ \ + \ \frac{1}{2}\sum_{z\in \Lambda_l}p(z)
E_{\nu_{\Lambda_l,j}}\left[g(\eta(0))\frac{\eta(0)-1}
{\eta(0)}\big(f(\eta^{0,z})-f(\eta)\big)^2\right].
\end{eqnarray*}

Let $W(l,j)$ and $W^{env}(l,j)$ be the inverse of the spectral gaps of
$\mathcal L_{\Lambda_l}$ and $L^{env}_{\Lambda_l}$ with respect to
$\Sigma_{\Lambda_l,j}$ and $\Sigma^*_{\Lambda_l,j}$ respectively.  In
particular, the following Poincar\'e inequalities are satisfied: For
all $L^2$ functions,
\begin{equation*}
\begin{split}
& \<f;f\>_{\mu_{\Lambda_l,j}} \;\leq\; W(l,j) D(\mu_{\Lambda_l,j},f)\\
&\quad \<f;f\>_{\nu_{\Lambda,l,j}}  \;\leq\; W^{env}(l,j)
D(\nu_{\Lambda_l,j},f)
\end{split}
\end{equation*}

In the next two lemmas, we do not assume that $g$ is increasing.  We
first relate the environment spectral gap to the untagged process
spectral gap.

\begin{lemma}
\label{spec_gap0}
Suppose on $\Sigma^*_{\Lambda_{l,j}}$ that $a_1^{-1}\leq
\eta(0)/g(\eta(0))\leq ja_0^{-1}$.  Then, for $j\geq 1$, we have that
$W^{env}(l,j) \leq (a_1a_0^{-1}j)^2 W(l,j-1)$.
\end{lemma}

\begin{proof}
Note $E_{\mu_\rho}[g(\eta(0))f(\eta)] =
\varphi(\rho)E_{\mu_\rho}[f'(\eta)]$ with $f'(\eta)=f(\eta+ \mf d_0)$,
where we recall $\mf d_0$ is the configuration with exactly one
particle at the origin.  By a suitable change of variables one
can show that $D(\mu_{\Lambda_l,j-1},f')\leq a_1a_0^{-1}j
D(\nu_{\Lambda_l,j},f)$.

By the assumption on $g$, for every $c\in\bb R$,
\begin{equation*}
\begin{split}
& E_{\nu_{\Lambda_l,j}}[(f - E_{\nu_{\Lambda_l,j}} f)^2]
\;\le \; \frac{E_{\mu_\rho}[\eta(0)(f-c)^2
\mb 1\{\Sigma^*_{\Lambda_l,j}\}]}
{E_{\mu_\rho}[\eta(0)\mb 1\{\Sigma^*_{\Lambda_l,j}\}]}\\
& \qquad \leq\; a_1 \, a_0^{-1} \, j\,
\frac{E_{\mu_\rho}[g(\eta(0))(f-c)^2
\mb 1\{\Sigma^*_{\Lambda_l,j}\}]}
{E_{\mu_\rho}[g(\eta(0))\mb 1\{\Sigma^*_{\Lambda_l,j}\}]}  \;\cdot
\end{split}
\end{equation*}
The change of variables $\eta'=\eta - \mf d_0$ and an appropriate
choice of the constant $c$ permits to rewrite last expression as
\begin{equation*}
a_1a_0^{-1}j \, E_{\mu_{\Lambda_l,j-1}}[(f'-E_{\mu_{\Lambda_l,j-1}}f')^2]
\;\leq\; a_1a_0^{-1}j \, W(l,j-1) D(\mu_{\Lambda_l,j-1},f')\;,
\end{equation*}
where the last inequality follows from the spectral gap for the zero
range process. By the observation made at the beginning of the proof,
this expression is bounded by
\begin{equation*}
(a_1\, a_0^{-1}\, j)^2 W(l,j-1) D(\nu_{\Lambda_l,j},f)\;,
\end{equation*}
which concludes the proof of the lemma.
\end{proof}

\begin{lemma}
\label{spec_gap}
Suppose $g$ satisfies $a_0\leq g(k)\leq a_1$ for $k\geq 1$, and
$\lim_{k\uparrow\infty}g(k) = L$.  For every $\alpha>0$, there is a
constant $B=B_\alpha$ such that $W(l,j)\leq B^l (1+\alpha)^j (l+j)^2$.
\end{lemma}

\begin{proof}
We need only establish, for all $L^2$ functions $f$, that
$$
\<f;f\>_{\mu_{\Lambda_l,j}}  \ \leq \ B^l (1+\alpha)^j (l+j)^2
D(\mu_{\Lambda_l,j},f).
$$

To argue the bound, we make a comparison with the measure
$\mu_{\Lambda_l,j}$ when $g(k)=\mb 1\{k\geq 1\}$.  Denote this measure
by $\mu^1_{\Lambda_l,j}$, and recall, by conversion to the simple
exclusion process (cf. \cite[Example 1.1]{LSV}, \cite{Morris}), that
\begin{equation}
E_{\mu^1_{\Lambda_l,j}}[(f-E_{\mu^1_{\Lambda_l,j}}f)^2] \
\leq \ b_0(l+j)^2D_{\mu^1_{\Lambda_l,j}}(f)
\label{specgap_1}
\end{equation}
for some finite constant $b_0$.  Write
\begin{equation*}
\begin{split}
& E_{\mu_{\Lambda_l,j}}\big[ (f-E_{\mu_{\Lambda_l,j}}f)^2\big]
\;=\; \inf_c E_{\mu_{\Lambda_l,j}} \big[ (f-c)^2\big] \\
& \qquad =\; \inf_c \frac{\sum_\eta \prod_{x=-l}^l
\frac{\varphi^{\eta(x)}}{g(\eta(x))!}(f(\eta)-c)^2
\mb 1\{\sum_x\eta(x)=j\}}{\sum_\eta  \prod_{x=-l}^l
\frac{\varphi^{\eta(x)}}{g(\eta(x))!}\mb 1\{\sum_x \eta(x)=j\}}
\;\cdot
\end{split}
\end{equation*}
Without loss of generality, we may now assume that $L=1$ since we can
replace $g$ by its scaled version, $g' = g/L$, in the above
expression.

For $\beta>0$, let $r_0$ be so large that $1-\beta \leq g(z)\leq
1+\beta$ for $z\geq r_0$.  Then,
$$
a_1^{-r_0}(1+\beta)^{-\eta(x)}\ \leq \frac{1}{g(\eta(x))!}
\ \leq \ a_0^{-r_0}(1-\beta)^{-\eta(x)}\; .
$$
This bound is achieved by overestimating the first $r_0$ factors by
the bound $a_0\mb 1\{z\geq 1\}\leq g(z)\leq a_1$, and the remaining
factors by
$$
(1+\beta)^{-\eta(x)}\ \leq \
\frac{1}{\prod_{z=r_0+1}^{\eta(x)} g(z)}\  \leq \ (1-\beta)^{-\eta(x)}
$$
where by convention an empty product is defined as $1$.

As there are $2l+1$ sites, we bound the right hand side of the
displayed expression appearing just below \eqref{specgap_1} by
\begin{equation*}
(a_0^{-1}a_1)^{(2l+1)r_0} \, [(1+\beta)(1-\beta)^{-1}]^{j} \,
E_{\mu^1_{\Lambda_l,j}} \big[(f- E_{\mu^1_{\Lambda_l,j}}f)^2\big]\;.
\end{equation*}
By the spectral gap estimate (\ref{specgap_1}) and the same bounds on
the Radon-Nikodym derivative of $d\mu^1_{\Lambda_l,j}/ d\mu_{\Lambda_l,j}$,
the previous expression is less than or equal to
\begin{eqnarray*}
(a_0^{-1}a_1)^{2(2l+1)r_0}[(1+\beta)(1-\beta)^{-1}]^{2j}
b_0 (l+j)^2D_{\mu_{\Lambda_l,j}}(f)\;.
\end{eqnarray*}
We may now choose $\beta=\beta(\alpha)$ appropriately to finish the
proof.
\end{proof}

We claim that for any constant $C>0$,
\begin{equation}
\label{l1}
\lim_{N\to\infty} \frac 1N \max_{1\le j\le C l \log N} W^{env}(l,j) \;=\; 0\;.
\end{equation}
Indeed, under the conditions (SL) this follows by Lemma
\ref{spec_gap0} and by assumption (SL3). On the other hand, under the
condition (B), by Lemma \ref{spec_gap} we may choose $\alpha$
appropriately to have $\max_{1\le j\le C l \log N} W(l,j) \le
C_2(\ell) N^{1/2}$. This proves \eqref{l1} in view of Lemma
\ref{spec_gap0}.

\section{``Global'' replacement}
\label{global_section}

In this section, we replace the full, or ``global'' empirical average
of a local, bounded and Lipschitz function, with respect to the
process $\eta_s$, in terms of the density field $\pi^N_s$.  The proof
involves only a few changes to the hydrodynamics proof of
\cite[Theorem 3.2.1]{dmp}, and is similar to that in \cite{jls}.
However, since the rate $g$ is bounded, some details with respect to
the ``$2$-blocks'' lemma below are different.

\begin{proposition}[``Global'' replacement]
\label{p1}
Let $r:\Omega_N \to \bb R$ be a local, bounded and Lipschitz function.
Then, for every $\delta>0$,
\begin{eqnarray*}
\limsup_{\varepsilon \to \infty} \limsup_{N \to \infty} \bb P^N \Big[
\int_0^T \frac{1}{N} \sum_{x\in {\bb T}_N} \tau_x
\V_{\varepsilon N}(\eta_s) \, ds \geq \delta \Big] =0,
\end{eqnarray*}
where
\begin{equation*}
\V_l(\eta) = \Big| \frac{1}{2l+1} \sum_{|y| \leq l} \tau_y r(\eta)
-\bar{r}(\eta^l(0)) \Big| {\rm \ \ and \ \ } \bar{r}(a) =
E_{\mu_a}[r]\;.
\end{equation*}
\end{proposition}

Denote by $\mc H(\mu|\nu)$ the entropy of $\mu$ with respect to $\nu$:
\begin{equation*}
\mc H(\mu|\nu) \;=\; \sup_{f} \Big\{ \int f d\mu \;-\; \log
\int e^f d\nu \Big\}\;,
\end{equation*}
where the supremum is over bounded continuous functions
$f$.

We may compute, with respect to the product measures
$\nu^N_{\rho_0(\cdot)}$ and $\nu_\rho$, that the initial entropy $\mc
H(\nu^N_{\rho_0(\cdot)}|\nu_\rho)\leq C_0N$ for some finite constant
$C_0$ depending only on $\rho_0(\cdot)$ and $g$. Let $f_t^N(\eta)$ be
the density of $\eta_t$ under $\bb P^N$ with respect to a reference
measure $\nu_\rho$ for $\rho>0$, and let $\hat{f}_t^N(\eta) = t^{-1}
\int_0^t f_s^N(\eta) ds$.  By usual arguments (cf. Section V.2
\cite{kl}),
\begin{equation*}
\mc H_N(\hat{f}_t^N):=\mc H(\hat{f}_t^Nd\nu_\rho|\nu_\rho) \leq C_0 N \quad
{\rm and} \quad \mc D_N(\hat{f}_t^N) := \Big \<\sqrt{\hat{f}_t^N}
(-L_N \sqrt{\hat{f}_t^N}) \Big\>_\rho \leq \frac{C_0}{N}\;,
\end{equation*}
where $\<u,v\>_\rho$ stands for the scalar product in $L^2(\nu_\rho)$,
as defined in the first section.

Consequently, to prove Proposition \ref{p1} it is enough to show, for
any finite constant $C$, that
\begin{equation}
\label{global_5_1}
\limsup_{\varepsilon \to 0} \limsup_{N \to
\infty} \sup_{\substack{\mc H_N(f) \leq C N \\ {\mathcal D}_N(f)
\leq C/N}} \int \frac{1}{N} \sum_{x\in {\bb T}_N} \tau_x \mc
V_{\varepsilon N} (\eta) f(\eta) d \nu_\rho =0
\end{equation}
where the supremum is with respect to $\nu_\rho$-densities $f$.

We may remove from the sum in (\ref{global_5_1}) the integers $x$
close to the origin, say $|x|\le 2 \varepsilon N$, as $\mc
V_{\varepsilon N}$ is bounded.  Now, the underlying reference measure
$\nu_\rho$ may be treated as homogeneous, and a standard strategy may
be employed as follows.

Proposition \ref{p1} now follows from the two standard lemmas below.
In this context, see also \cite{dmp}, and \cite{kl} where the same
method is used to prove \cite[Theorem 3.2.1]{dmp} and \cite[Lemma
V.1.10]{kl} respectively.

\begin{lemma}[Global 1-block estimate]
\label{g3}
\begin{equation*}
\limsup_{k \to \infty} \limsup_{N \to \infty}
\bb E^N \Big [ \int_0^T\frac{1}{N}
\sum_{|x| >  2\varepsilon N} \tau_x \V_k (\eta_s)\, ds \Big]\  =\ 0\;.
\end{equation*}
\end{lemma}

The proof of Lemma \ref{g3} is the same as for \cite[Lemma 5.2]{jls},
and follows the scheme of \cite[Lemma V.3.1]{kl}, using that $g$ has
``sub-linear growth (SLG)''.  Details are omitted here.

\begin{lemma}[Global 2-block estimate]
\label{g4}
\begin{eqnarray*}
&&\limsup_{k\rightarrow \infty}\limsup_{\varepsilon \rightarrow 0}
\limsup_{N\rightarrow \infty} \\
&&\ \ \ \ \ \qquad
\bb E^N \Big [\int_0^T\frac{1}{2N\varepsilon +1} \sum_{|y|\leq
  N\varepsilon}
\frac{1}{N} \sum_{|x| > 3\varepsilon N} |\eta_s^k(x+y) -
\eta_s^k(x)| \, ds \Big] \ =\  0\;.
\end{eqnarray*}
\end{lemma}

\begin{proof}
We discuss in terms of modifications to the argument in \cite[Section
V.4]{kl}.  The first step is to cut-off high densities. We claim that
\begin{equation*}
\limsup_{A\rightarrow \infty}\limsup_{k\rightarrow \infty}
\limsup_{N\rightarrow \infty} \bb E^N \Big [ \int_0^T
\frac{1}{N} \sum_{|x| > 3\varepsilon N} \eta_s^k(x)
\mb 1 \{\eta_s^k(x) >A\} \, ds \Big]\  =\  0\;.
\end{equation*}

To prove this assertion, we first replace the sum over $x$ by a sum
over all sites of $\bb T_N$. At this point, since the environment at
time $\eta_t$ is obtained from the system by a shift, we may replace
the variable $\eta_t$ by $\xi_t$. We need therefore to estimate
\begin{equation*}
\mb E_{\mu^N_{\rho_0(\cdot)}} \Big[ \frac{\xi_0(0)}{\rho_0(0)}
\int_0^T \frac{1}{N}\sum_{x\in \bb T_N}
\xi_s^k(x) \mb 1 \{\xi_s^k(x) >A\} \, ds \Big ]\;.
\end{equation*}
Let $\bar\rho = \|\rho_0\|_{L^\infty}$, and note that
$\mu_{\rho_0(\cdot)}$ is stochastically dominated by $\mu_{\bar\rho}$.
By attractiveness we may replace $\mu^N_{\rho_0(\cdot)}$ by
$\mu_{\bar\rho}$ in the previous expression and bound this expectation
by
\begin{eqnarray*}
\mb E_{\mu_{\bar\rho}}\Big [\frac{\xi_0(0)}{\rho_0(0)} \int_0^T
\frac{1}{AN} \sum_{x\in \bb T_N} (\xi_s^k(x))^2  \,ds \Big ]\;.
\end{eqnarray*}
By Schwarz inequality, and noting that $\mu_{\bar\rho}$ is invariant
with respect to the untagged process $\xi_s$, the last expression is
of order $A^{-1}$, which proves the claim.

In view of the truncation just proved and the entropy calculations
presented at the beginning of this section, to prove the lemma it is
enough to show that for every $A>0$,
\begin{equation*}
\limsup_{k \to \infty} \limsup_{N \to \infty}
\sup_{\substack{\mc H_N(f) \leq C N \\ \mc
    D_N(f) \leq C/N}} \int \frac{1}{2N\varepsilon +1}
\sum_{|y|\leq  N\varepsilon} \frac{1}{N}
\sum_{|x| > 3\varepsilon N} W^{k,A}_{x,y}(\eta) f(\eta)\, d\nu_\rho \ =\  0\;.
\end{equation*}
where
\begin{equation*}
W^{k,A}_{x,y}(\eta) \;=\; |\eta^k(x+y) - \eta^k(x)| \,
\mb 1 \big\{ \max\{\eta^k(x), \eta^k(x+y)\} \leq A\}\;.
\end{equation*}
The argument is now the same as in the proof of Lemma 5.3 \cite{jls}
following \cite[Section V.5]{kl}.
\end{proof}

\section{``Local'' one-block estimate}
\label{1_block_section}
We now detail a ``local'' one-block limit.  Let $h:\N_0\rightarrow \R$
be a bounded, Lipschitz function, and $H(a)= E_{\nu_a}[h(\eta(0))]$.
Define also
\[
V_l(\eta) = h(\eta(0)) - H(\eta^l(0)).
\]

\begin{lemma}[One-block estimate]
\label{l3}
For every $0\le t\le T$,
\begin{equation*}
\limsup_{l \to \infty} \limsup_{N \to \infty} \bb E^N  \Big[\,
\Big|\int_0^t V_l(\eta_s) \, ds \Big|\, \Big ] =0\;.
\end{equation*}
\end{lemma}

\begin{proof}  The proof is in four steps.
\vskip .1cm

{\it Step 1.}  The first step is to introduce a truncation.  Since the
dynamics is not attractive, we cannot bound $\eta(0)> A$ for some
constant $A$ in a simple way.  However, by considering the maximum of
such quantities over the torus, we may rewrite the maximum in terms of
the original system $\xi_s$, which is attractive:
$$
\max_{x\in \T_N} \eta_s(x) \ = \ \max_{x\in \T_N} \xi_s(x).
$$
Also, by simple estimates, recalling $\bar{\rho}=\|\rho_0\|_L^\infty$,
we have that
\begin{eqnarray*}
\mb P_{\nu_{\rho_0(\cdot)}} \Big[ \max_x \xi_s(x)  \geq  C\log N \Big]
&= & \mb E_{\mu_{\rho_0(\cdot)}}\Big[ \frac{\xi_0(0)}{\rho_0(0)}
\, \mb 1\{\max_x \xi_s(x)\geq C\log N\}\Big] \\
&\le & \mb E_{\mu_{\bar\rho}}\Big[ \frac{\xi_0(0)}{\rho_0(0)}
\, \mb 1\{\max_x \xi_s(x)\geq C\log N\}\Big]\;.
\end{eqnarray*}
Under the stationary measure $\mu_{\bar\rho}$, the variables $\xi_s(x)$
are independent and identically distributed, with finite exponential
moments of some order.  Hence, by Chebychev's inequality, the last
expression vanishes as $N\uparrow\infty$ for a well chosen constant
$C=C_1$.  Therefore, as $\eta^l(0)\leq \max_x \eta(x)$, it is enough to
estimate
$$
\bb E^N  \Big[\, \Big|\int_0^t V_l(\eta_s)\mb 1\{G_{N,l}\}(\eta_s) \,
ds \Big|\, \Big ].
$$
where $G_{N,l} = \{\eta: \eta^l(0)\leq  C_1\log N\}$.
\vskip .1cm

{\it Step 2.}
Since the initial entropy $\mc H(\nu_{\rho_0(\cdot)}^N|\nu_\rho)$ is
bounded by $C_0 N$, by the entropy inequality,
\begin{eqnarray*}
&&\bb E^N \Big[\, \Big| \int_0^t V_l(\eta_s)\mb 1\{G_{N,l}\}(\eta_s)
\, ds\Big| \, \Big]\\
&&\ \ \ \ \ \ \ \;\leq\; \frac{C_0}{\gamma} + \frac{1}{\gamma N}
\log \bb E_{\nu_\rho} \Big[\exp\Big\{\gamma N \Big|\int_0^t
V_l(\eta_s)\mb 1\{G_{N,l}\}(\eta_s) \, ds \Big| \Big\} \Big]\; .
\end{eqnarray*}
We can get rid of the absolute value in the previous integral, using
the inequality $e^{|x|} \leq e^x+e^{-x}$. By Feynman-Kac formula, the
second term on the right hand side is bounded by $(\gamma N)^{-1} T
\lambda_{N,l}$, where $\lambda_{N,l}$ is the largest eigenvalue of
$N^2 L_N + \gamma N V_l \mb 1\{G_{N,l}\}$.  Therefore, to prove the
lemma, it is enough to show that $(\gamma N)^{-1} \lambda_{N,l}$
vanishes, as $N\uparrow\infty$ and then $l\uparrow\infty$, for every
$\gamma>0$.  \vskip .1cm

{\it Step 3.}
By the variational formula for $\lambda_{N,l}$,
\begin{equation}
\label{ec1}
(\gamma N)^{-1} \lambda_{N,l} \;=\;
\sup_f \Big\{ \< V_l \mb 1\{G_{N,l}\} \, f^2 \>_\rho
- \gamma^{-1} N \< f(-L_N f) \>_\rho  \Big\}\;,
\end{equation}
where the supremum is carried over all densities $f^2$ with respect to
$\nu_\rho$.  As the Dirichlet forms satisfy $\<f(-L^{env}_{\Lambda_l}
f)\>_\rho \leq \<f(-L_N f)\>_\rho$ (cf. \cite[equation (3.1)]{Szrtg}),
we may bound the previous expression by a similar one where $L_N$ is
replaced by $L^{env}_{\Lambda_l}$.

Denote by $\hat f^2_{l}$ the conditional expectation of $f^2$ given
$\{\eta(z) : z\in\Lambda_l\}$.  Since $V_l \mb 1\{G_{N,l}\}$ depends
on the configuration $\eta$ only through $\{\eta(z) : z\in\Lambda_l\}$
and since the Dirichlet form is convex, the expression inside braces
in \eqref{ec1} is less than or equal to
\begin{equation}
\label{c2}
\int  V_l \mb 1\{G_{N,l}\} \, \hat f^2_{l} \, d \nu^{\Lambda_l}_\rho \;
-\; \gamma^{-1} N \int \hat f_{l} \,
(-L^{env}_{\Lambda_l} \hat f_{l}) \, d \nu^{\Lambda_l}_\rho \;.
\end{equation}

The first term in this formula, decomposing in terms of canonical
measures $\nu_{\Lambda_l,j}$, is equal to
\begin{equation*}
\sum_{j=1}^{C_1 l \log N} c_{l,j} (f) \int V_l \mb 1\{G_{N,l}\} \, \hat f^2_{l,j}
\, d \nu_{\Lambda_l, j}
\end{equation*}
where the value of the constant $C_1$ changed and
\begin{equation*}
c_{l,j} (f) \;=\; \int_{\Sigma_{\Lambda_l, j}} \hat f^2_{l}
\, d \nu^{\Lambda_l}_\rho  \;, \quad
\hat f^2_{l,j} (\eta) \;=\; c_{l,j}(f)^{-1} \,
\nu^{\Lambda_l}_\rho (\Sigma_{\Lambda_l, j})\, \hat f^2_{l} (\eta)\;.
\end{equation*}
The sum starts at $j=1$ because there is always a particle at the
origin.  Note also that $\sum_{j\ge 1} c_{l,j} (f) =1$ and that $ \hat
f^2_{l,j} (\cdot)$ is a density with respect to $\nu_{\Lambda_l, j}$.

Also, the Dirichlet form term of \eqref{c2} can be written
as
\begin{equation*}
\gamma^{-1} N \sum_{1\leq j\leq C_1 l \log N} c_{l,j} (f)
\int  \hat f_{l,j} \, (-L^{env}_{\Lambda_l}
\hat f_{l,j}) \, d \nu_{\Lambda_l, j} \;.
\end{equation*}
In view of this decomposition, \eqref{ec1} is bounded above by
\begin{equation*}
\sup_{1\leq j \leq C_1 l \log N} \sup_{f}  \Big\{ \int V_l  \, f^2 \,
d \nu_{\Lambda_l, j}
\;-\; \gamma^{-1} N  \int  f \,  (-L^{env}_{\Lambda_l} f)
\, d \nu_{\Lambda_l, j} \Big\}\;,
\end{equation*}
where the second supremum is over all densities $f^2$ with respect
to $\nu_{\Lambda_l, j}$.
\vskip .1cm

{\it Step 4.}  Recall that $V_l(\eta) = h(\eta(0)) - H (\eta^l(0))$.
Let $V_{l,j}(\eta) = V_l - E_{\nu_{\Lambda_l,j}}[V_l]$.  By
\eqref{l1}, $N^{-1} \max_{1\leq j \leq C_1 l \log N} W^{env}(l,j)$
vanishes as $N\uparrow\infty$.  Then, as $h$ is bounded, by Rayleigh
expansion \cite[Theorem A3.1.1]{kl}, for $j\leq C_1 l \log N$ and
sufficiently large $N$,
\begin{eqnarray*}
&& \int V_{l}  \, f^2 \, d \nu_{\Lambda_l, j}
\;-\; \gamma^{-1} N  \int  f \,  (-L^{env}_{\Lambda_l} f)
\, d \nu_{\Lambda_l, j} \\
&& \qquad
\le\; \int V_l \, d\nu_{\Lambda_l,j} + \frac{\gamma N^{-1}}{1-2\|V_l\|_{L^\infty}
W^{env}(l,j)\gamma N^{-1}} \int V_{l,j} (-L^{env}_{\Lambda_l})^{-1} V_{l,j}
\, d \nu_{\Lambda_l, j}\\
&& \qquad
\le\; \int V_l \, d\nu_{\Lambda_l,j} + 2 \gamma N^{-1}
\int V_{l,j} (-L^{env}_{\Lambda_l})^{-1} V_{l,j} \, d \nu_{\Lambda_l, j}\; .
\end{eqnarray*}

The second term is bounded as follows.  By the spectral theorem,
second term then is less than or equal to
\begin{equation}
\label{last_line_1block}
2W^{env}(l,j)\gamma N^{-1} \int V_{l,j}^2 \, d \nu_{\Lambda_l, j}\;\le\;
8\|h\|^2_{L^\infty}W^{env}(l,j) \gamma N^{-1} \;.
\end{equation}
This expression vanishes as $N\uparrow\infty$ in view of \eqref{l1}.

On the other hand, the first term is written as
$$
\int V_{l}\, d\nu_{\Lambda_l,j} \ = \ \int h(\eta(0)) \, d
\nu_{\Lambda_l, j} - H(j/2l+1)\;.
$$
By Lemma \ref{lclt1} below, this difference vanishes uniformly in $j$
as $l\uparrow\infty$.  This proves that \eqref{ec1} vanishes as
$N\uparrow\infty$ and then $l\uparrow\infty$, finishing the proof.
\end{proof}

\begin{lemma}
\label{lclt1}
Let $h:\N_0\rightarrow \R$ be a bounded Lipschitz function which
vanishes at infinity.  Then, we have
\begin{equation*}
\limsup_{l\rightarrow \infty} \sup_{k\geq 1} \Big |
E_{\nu_{\Lambda_l,k}}[h(\eta(0))] -
E_{\nu_{k/|\Lambda_l|}}[h(\eta(0))] \Big | \ = \ 0\; .
\end{equation*}
\end{lemma}

\begin{proof}
The argument is in three steps.

\vskip .1cm {\it Step 1.}  We first consider the case $1\leq k\leq
K_0$.  By adding and subtracting $h(1)$, we need only estimate
\begin{equation}
\label{l2}
|E_{\nu_{\Lambda_l,k}}[h(\eta(0))] - h(1)| \ \ {\rm and \ \ }
|E_{\nu_{k/|\Lambda_l|}}[h(\eta(0))] -h(1)|\; .
\end{equation}
The first term is bounded by $2\|h\|_{L^\infty} \nu_{\Lambda_l,k}
\{\eta(0)\ge 2\}$. To show that it vanishes as $l\uparrow\infty$, note
that $\eta(0)\le k$ and that $E_{\mu_{\Lambda_l,k}}[\eta(0)] =
k/(2l+1)$ to write
\begin{eqnarray*}
\nu_{\Lambda_l,k} \{\eta(0)\ge 2\} &=&
\frac{1}{E_{\mu_{\Lambda_l,k}}[\eta(0)]}
E_{\mu_{\Lambda_l,k}}[\eta(0) \mb 1\{\eta(0)\ge 2\}] \\
&\leq &  (2l+1) \mu_{\Lambda_l,k} \{\eta(0)\ge 2\}\; .
\end{eqnarray*}
For $2\leq s\leq k$, we may write the canonical measure in terms of
the grand canonical:
\begin{equation*}
\mu_{\Lambda_l,k} \{\eta(0)= s\} \;=\;
\mu_{\rho}\{\eta(0)=s\}\,
\frac{\mu_{\rho}\{ \sum_{0<|x|\leq l}\eta(x) = k-s\}}
{\mu_{\rho}\{\sum_{|x|\leq l}\eta(x) = k\}}
\end{equation*}
for any choice of the parameter $\rho$. Recall $\mu^1_{\Lambda_l,j}$
is the canonical measure when $g(k)=\mb 1\{k\geq 1\}$. In the
numerator and the denominator, at least $2\ell - k$ sites receive no
particles. We may therefore replace in these sites the rate $g$ by the
rate constant equal to one with no cost. Since $a_0\le g(\ell)\le
a_1\ell$, in the remaining sites we have that $C(k)^{-1} \le a_0^n \le
g(n)! \le a_1^n n!  \le C(k)$ if $n\le k$. The previous expression is
thus bounded above by
\begin{equation*}
C(k) \, \mu^1_{\Lambda_l,k} \{\eta(0)=s\} \;=\;
C(k)\, \left(\begin{array}{c}2l\\k-s\end{array}\right)
\Big /\left(\begin{array}{c}2l+1\\k\end{array}\right) \ = \
O(l^{-s})\;.
\end{equation*}

To bound the second term in \eqref{l2}, we proceed in a similar way.
The absolute value of the difference $E_{\nu_{\rho}}[h(\eta(0))]
-h(1)$ is bounded by $2 \, \Vert h\Vert_\infty \nu_{\rho}\{\eta(0)\ge
2\}$. Last probability is equal to $\rho^{-1} E_{\mu_{\rho}} [
\eta(0)\, \mb 1\{ \eta(0)\ge 2\} \,]$. Since $g(n)\ge a_0$, change of
variables $\eta'= \eta-2\mf d_0$ permits to bound the previous
expression by $C_0 \varphi(\rho)^2 [\rho+2]/\rho$ for some finite
constant $C_0$. Since $g(n) \le a_1n$, $\varphi(\rho) \le a_1\rho$. In
conclusion, the second term in \eqref{l2} is bounded above by $C_0 \,
\Vert h\Vert_\infty \, (k/l)^2$, which concludes the proof of Step 1.
\vskip .1cm

{\it Step 2.}  Next, we consider the case in which $K_0 \le k \le B
|\Lambda_l|$ for some $B<\infty$. By definition of the Palm measure,
the difference $E_{\nu_{\Lambda_l,k}} [h(\eta(0))] -
E_{\nu_{k/|\Lambda_l|}} [h(\eta(0))]$ is equal to
\begin{equation*}
\frac {|\Lambda_l|}k\,
\Big\{E_{\mu_{\Lambda_l,k}} \big[\eta(0)\, h(\eta(0))\big] -
E_{\mu_{k/|\Lambda_l|}} \big [\eta(0)\, h(\eta(0))\big] \Big\}\;.
\end{equation*}
By \cite[Corollary 1.7, Appendix 2.1]{kl}, this expression is bounded
above by $C_0 k^{-1}$ for some finite constant $C_0$. This expression
can be made as small as need by choosing $K_0$ large.  \vskip .1cm

{\it Step 3.}  Finally, we consider the case $k \geq B
|\Lambda_l|$. We shall take advantage of the fact that $h$ vanishes at
infinity. Fix $A>0$ to bound $E_{\nu_{\Lambda_l,k}}[h(\eta(0))]$ by
\begin{equation*}
E_{\nu_{\Lambda_l,k}} \big[h(\eta(0)) \, \mb 1\{\eta(0)\leq A\} \big]
\;+\; \sup_{x\geq  A} h(x)
\end{equation*}
By definition of the Palm measure and since the density
$k/|\Lambda_l|$ is bounded below by $B$, the first term is less than
or equal to
\begin{equation*}
\frac{\|h\|_{L^\infty}\, |\Lambda_l|}{k}
E_{\mu_{\Lambda_l,k}} \big[ \eta(0)
\, \mb 1\{\eta(0)\leq A\} \big]
\;\leq\; \frac{A \|h\|_{L^\infty}}{B}\;.
\end{equation*}
In view of the previous estimates, we see that the expectation
$E_{\nu_{\Lambda_l,k}}[h(\eta(0))]$ can be made arbitrarily small by
choosing $A$ and $B$ sufficiently large.  The expectation
$E_{\nu_{k/|\Lambda_l|}}[h(\eta(0)]$ can be estimated similarly.
\end{proof}

\section{Local two-blocks estimate}
\label{2_block_section}

In this section we show how to go from a box of size $l$ to a box of
size $\epsilon N$.

\begin{lemma}[Two-blocks estimate]
\label{l4}
Let $H:\R_+\rightarrow \R$ be a bounded, Lipschitz function, which
vanishes at infinity, $\lim_{x\to \infty}H(x)=0$.  Then, for every
$t>0$,
\begin{equation}
\label{ec2}
\limsup_{l \to \infty} \limsup_{\epsilon \to 0} \limsup_{N \to \infty}
\bb E^N \Big[\, \Big| \int_0^t \big\{H(\eta_s^l(0)) -\frac{1}{\epsilon
  N} \sum_{x =1} ^{\epsilon N} H(\eta_s^l(x)) \big\} ds \Big| \, \Big]
= 0.
\end{equation}
\end{lemma}

\begin{proof}
The proof is handled in several steps.  \vskip .1cm

\noindent {\it Step 1.}  As $H$ is bounded, the expectation in
(\ref{ec2}) is bounded
\begin{equation*}
\frac{1}{\epsilon N} \sum_{x=4l+2} ^{\epsilon N} \bb E^N
\Big[ \, \Big| \int_0^t \big\{H(\eta_s^l(0)) - H(\eta_s^l(x)) \big\} ds
\Big| \, \Big] \;+\; \frac{C_0 \|H\|_{L^\infty} \, l}{\epsilon N}
\end{equation*}
for some finite constant $C_0$.  Hence, we need to estimate, uniformly
over $4l+2\leq x\leq \epsilon N$,
\begin{equation*}
\bb E^N \Big[ \, \Big| \int_0^t \big\{H(\eta_s^l(0))
- H(\eta_s^l(x)) \big\} \, ds \Big| \, \Big]\;.
\end{equation*}
\vskip .1cm

\noindent {\it Step 2.}
Write
$$
H(\eta^l(0)) - H(\eta^l(x)) \ = \
H(\eta^l(0))- H(\eta^l(2l+1)) + H(\eta^l(2l+1))- H(\eta^l(x))\;.
$$
We now claim that
$$
\lim_{l\to\infty}\lim_{N\to\infty}
E_N \Big[ \, \Big | \int_0^t \big\{
H(\eta^l_s(0))- H(\eta^l_s(2l+1)) \big\}\, ds \Big |  \,\Big] \ = \ 0\; .
$$
Indeed, since $H(\eta^l(0))- H(\eta^l(2l+1))$ is a function of
$\hat\Lambda_{l}=\{-l,\ldots,3l+1\}$, we may apply the ``local
$1$-block'' argument for Lemma \ref{l3} up to
(\ref{last_line_1block}), with respect to $V'_l = H(\eta^l(0))-
H(\eta^l(2l+1))$.  Now, in the last line of the proof of Lemma
\ref{l3}, instead of using Lemma \ref{lclt1}, we use Lemma \ref{lclt2}
to show the expectation under the canonical measure
$\nu_{\hat\Lambda_{l},k}$ vanishes, $\lim_{l\uparrow\infty}\sup_{k\geq
  1}E_{\hat\Lambda_{l},k}[V'_l] = 0$. \vskip .1cm

\noindent {\it Step 3.}
Therefore, we need only estimate when the integrand is
$H(\eta^l(2l+1))-H(\eta^l(x))$.  As for the ``local $1$-block''
development (Lemma \ref{l3}), we may introduce a truncation, and
restrict to the set $G_{N,l,x}=\{\eta: \eta^l(2l+1) + \eta^l(x) \leq
2C_1\log N\}$.  That is, we need only bound, uniformly over $x$,
$$
\bb E^N \Big[ \, \Big| \int_0^t \big[ H(\eta_s^l(2l+1))
- H(\eta_s^l(x)) \big ]\, \mb 1\{G_{N,l,x}\} \, ds
\Big| \, \Big] \;\cdot
$$
\vskip .1cm

\noindent {\it Step 4.}  Following the first part of the proof of
Lemma \ref{l3}, appealing to entropy estimates and eigenvalue
estimates, we need only to bound, uniformly in $4l+2\leq x\leq
\epsilon N$,
\begin{equation}
\sup_{f} \Big\{ \big\< [H(\eta^l(2l+1)) -H(\eta^l(x))]
\mb 1\{G_{N,l,x}\} \,,\, f^2 \big\>_\rho - N \gamma^{-1} \< f,(-L_N
f)\>_\rho\Big\},
\label{2blocks_step4}
\end{equation}
where the supremum is over all density functions $f^2$ with $\int f^2
d\nu_\rho=1$.

Since $V_{l,x}(\eta) = H(\eta^l(2l+1)) -H(\eta^l(x))$ does not involve
the origin, we can avoid details involving the inhomogeneity at point
$0$ in the following.  Define disjoint blocks $\Lambda'_l =
\{l+1,\ldots, 3l+1\}$ and $\Lambda_l(x) = \{x-l,\ldots, x+l\}$.  Let
$L_{\Lambda_{l,x}}$ be the restriction of $L^{env}_N$ to the set
$\Lambda_{l,x}=\Lambda'_l \cup \Lambda_l(x)$, and define also
$L_{l,x}$ by
\begin{eqnarray*}
L_{l,x} f(\eta) &= &  \frac{1}{2}\,
g(\eta(x-l)) \, [f(\eta^{x-l,3l+1}) - f(\eta)]\\
&+& \frac{1}{2} \, g(\eta(3l+1))\,
[f(\eta^{3l+1,x-l})-f(\eta)]\;.
\end{eqnarray*}
The operator $L_{l,x}$ corresponds to zero-range dynamics where
particles jump between endpoints $3l+1$ and $x-l$.

As $x \leq \epsilon N$, by adding and subtracting at most $\epsilon N$
terms (cf. \cite[p. 94-95]{kl}, \cite[equation (3.1)]{Szrtg}), we have
that
\begin{equation*}
\<f(- L_{l,x} f)\>_\rho \ \leq\ \epsilon N\<f(-L^{env}_N f)\>_\rho \;.
\end{equation*}
Hence,
$$
\Big\<f, - \big(N\gamma^{-1}L_{\Lambda_{l,x}} +
\epsilon^{-1}\gamma^{-1}L_{l,x} \big) f \Big\>_\rho
\ \leq \ 2N \gamma^{-1} \< f(-L_N f)\>_\rho \;,
$$
and we may replace $N\gamma^{-1}L_N$ in (\ref{2blocks_step4}) by
$(1/2) (N\gamma^{-1}L_{\Lambda_{l,x}} + \epsilon^{-1}\gamma^{-1}L_{l,x})$.
\vskip .1cm

\noindent {\it Step 5.}
To simplify notation, we shift the indices so that the blocks are to
the left and right of the origin.  In particular, let $\Lambda^-_l =
\{-(2l+1), \dots, -1\}$, $\Lambda^+_l = \{1, \dots, (2l+1)\}$ and
$\Lambda^*_l = \Lambda^-_l \cup \Lambda^+_l$.  Configurations of $\bb
N_0^{\Lambda^-_l}$ will be denoted by the Greek letter $\eta$, while
configurations of $\bb N_0^{\Lambda^+_l}$ are denoted by the Greek
letter $\zeta$.  Recall $\mf d_z$ stands for the configuration with no
particles but one at $z$.

Consider the generator $L_{N,\varepsilon,l}$ with respect to $\bb
N_0^{\Lambda^*_l}$, $L_{N, \varepsilon, l} = N L^-_{l} + N L^+_{l} +
\varepsilon^{-1} L^0_{l}$.  Here,
\begin{equation*}
\begin{split}
& (L^-_{l} f) (\eta, \zeta) \;=\; \sum_{x, y\in
    \Lambda_l^-}
p(y-x) \, g(\eta(x)) \, [f(\eta^{x,y}, \zeta)-f(\eta, \zeta)] \; , \\
& \quad (L^+_{l} f) (\eta, \zeta) \;=\; \sum_{x,
  y\in \Lambda_l^+} p(y-x) \, g(\zeta(x)) \, [f(\eta,
\zeta^{x,y})-f(\eta, \zeta)]\; , \\
& \qquad (L^0_{l} f) (\eta, \zeta) \;=\;
(1/2)\, g(\eta(-1)) \, [f(\eta - \mf d_{-1},
\zeta + \mf d_1)-f(\eta, \zeta)] \\
& \qquad\qquad\qquad\qquad\quad
\;+\;  \,  (1/2)\, g(\zeta(1)) \, [f(\eta + \mf d_{-1},
\zeta- \mf d_{1})-f(\eta  , \zeta)]  \; .
\end{split}
\end{equation*}
Note that inside each set $\Lambda^\pm_l$ particles jump at rate
$N$ while jumps between sets are performed at rate
$\varepsilon^{-1}$.

Recall $\mu^{\Lambda^-_l}_\rho$, $\mu^{\Lambda^+_l}_\rho$,
$\mu^{\Lambda^*_l}_\rho$ are the restrictions of $\mu_\rho$ to $\bb
N_0^{\Lambda^-_l}$, $\bb N_0^{\Lambda^+_l}$, $\bb N_0^{\Lambda^*_l}$,
respectively.  The Dirichlet forms associated to the generators
$L^-_{l}$, $L^+_{l}$, $L^0_{l}$ are given by
\begin{equation}
\label{fl03}
\begin{split}
& D_{\Lambda^-_l}(\mu^{\Lambda^*_l}_\rho , f) \; =\;
\< f , (- L^-_{l} f) \>_{\mu^{\Lambda^*_l}_\rho} \; , \quad
D_{\Lambda^+_l}(\mu^{\Lambda^*_l}_\rho , f) \; =\;
\< f , (- L^+_{l} f) \>_{\mu^{\Lambda^*_l}_\rho} \; , \\
&\qquad D_{0}(\mu^{\Lambda^*_l}_\rho , f) \; =\;
\< f , (- L^0_{l} f) \>_{\mu^{\Lambda^*_l}_\rho} \; \cdot
\end{split}
\end{equation}
A simple computation shows that the Dirichlet form can be written as
\begin{equation*}
\begin{split}
& D_{\Lambda^-_l}(\mu^{\Lambda^*_l}_\rho , f) \;=\;
\frac{\varphi(\rho)}{2} \sum_{x=-(2l+1)}^{-2}
\int \{ f(\eta + \mf d_{x+1}, \zeta)
- f (\eta + \mf d_{x}, \zeta) \}^2\, \mu^{\Lambda^*_l}_\rho
(d\eta, d\zeta) \;, \\
&\qquad
D_{0}(\mu^{\Lambda^*_l}_\rho , f) \; =\;
\frac{\varphi(\rho)}{2}
\int \{ f(\eta + \mf d_{-1}, \zeta)
- f (\eta , \zeta + \mf d_{1}) \}^2\, \mu^{\Lambda^*_l}_\rho
(d\eta, d\zeta)\; .
\end{split}
\end{equation*}

In this notation, it will be enough, with respect to equation
(\ref{2blocks_step4}), to bound for $a>0$ the quantity
\begin{equation}
\sup_{f} \Big\{ \big\< [H(\eta^l)-H(\zeta^l)]
\, \mb 1\{G'_{N,l}\},f^2 \big\>_\rho -
a\<f, (-L_{N,\varepsilon, l} f)\>_\rho\Big\}\;,
\label{step5_eqn}
\end{equation}
where $\eta^l = (2l+1)^{-1} \sum_{x\in \Lambda^-_l}\eta(x)$, $\zeta^l
= (2l+1)^{-1} \sum_{x\in \Lambda^+_l}\zeta(x)$, and $G'_{N,l} =
\{(\eta,\zeta): \eta^l+\zeta^l\leq 2C_1\log N\}$.  By convexity of the
Dirichlet form, as in the proof of Lemma \ref{l3}, the supremum may be
taken over functions $f$ on $\bb N_0^{\Lambda^*_l}$ such that
$\<f^2\>_{\mu_\rho^{\Lambda^*_l}}=1$, and the measure $\mu_\rho$ in
(\ref{step5_eqn}) may be replaced by $\mu^{\Lambda^*_l}_\rho$.  \vskip
.1cm

\noindent {\it Step 6.}
This quantity is estimated in three parts.  The first part restricts
to the set $S^1_{N,l} = \{(\eta,\zeta): \eta^l + \zeta^l \leq B\}$ for
some $B$ fixed.  In this case, where we have truncated at a fixed
level $B$, we can use the ``local $1$-block'' method of Lemma \ref{l3}
to show that
$$
\sup_f\Big\{ \big \<\big [H(\eta^l)
- H(\zeta^l) \big]\, \mb 1\{S^1_{N,l}\} \,,\,
f^2\big\>_{\mu^{\Lambda^*_l}_\rho} \ - \
a\<f,(- L_{N,\varepsilon,l} f)\>_{\mu^{\Lambda^*_l}_\rho} \Big\}
$$
vanishes as $N\uparrow\infty$, $\epsilon\downarrow 0$ and then
$l\uparrow\infty$.

Indeed, by convexity considerations, we can decompose the expression
in braces in terms of canonical measures $\mu_{\Lambda^*_l,k}$
concentrating on $k$ particles in $\Lambda^*_l$. Since for the
generator $L_{N,\varepsilon,l}$ jumps are speeded up by $N$ inside
each cube,
\begin{equation*}
\begin{split}
& \big \< \big [H(\eta^l) - H(\zeta^l) \big] \,
\mb 1\{S^1_{N,l}\} \,,\, f^2 \big\>_{\mu_{\Lambda^*_l,k}}\  - \
a \<f,(- L_{N,\varepsilon,l} f) \>_{\mu_{\Lambda^*_l,k}} \\
& \quad \le\; \big \<\big[ H(\eta^l)
- H(\zeta^l) \big] \, \mb 1\{S^1_{N,l}\} \,,\,
f^2\big\>_{\mu_{\Lambda^*_l,k}} \  - \
a\epsilon^{-1} \<f,-(L_l^- + L_l^+ + L^0_l f) \>_{\mu_{\Lambda^*_l,k}}
\end{split}
\end{equation*}
Let $\tilde V_l = [H(\eta^l) - H(\zeta^l)]\mb 1\{S^1_{N,l}\}$. Note
that $\tilde V_l$ has mean-zero with respect to $\mu_{\Lambda^*_l,k}$
and that $\|\tilde V_l\|_{L^\infty}\leq 2\|H\|_{L^\infty}$. By the
Rayleigh estimate \cite[Theorem A3.1.1]{kl} and by the spectral gap,
for $k\leq 2(2l+1)B$, the previous expression is bounded above by
\begin{equation*}
\begin{split}
& \frac{a^{-1} \epsilon}{1-4\|H\|_{L^\infty} W^*(l,k) \, a^{-1} \epsilon}
\int \tilde V_l\, (-L^-_l-L^+_l-L^0_l)^{-1} \, \tilde V_l \,
d\mu_{\Lambda^*_l,k}\\
& \quad \leq \ 2 a^{-1} \epsilon W^*(l,k)
\int \tilde V^2_l \, d\mu_{\Lambda^*_l,k} \;,
\end{split}
\end{equation*}
where $W^*(l,k)$ is the inverse of the spectral gap of $L^-+L^+
+L^0_l$ with respect to the process on $\Lambda^*_l$ with $k$
particles. As $\epsilon\downarrow 0$, the previous expression
vanishes. \vskip .1cm

\noindent {\it Step 7.}
The second part now restricts to $S^2_{N,l} = \{(\eta,\zeta): \zeta^l
\geq A, \eta^l \geq A\}$ for some constant $A$.  On this event, the
sum $H(\eta^l) +H(\zeta^l)$ is absolutely bounded by $2 \sup_{z\geq
  A} |H(z)|$ so that
$$
\big \< \big [H(\eta^l) - H(\zeta^l) \big] \, \mb 1\{S^2_{N,l}\} \,,\,
f^2 \big\>_{\mu^{\Lambda^*_l}_\rho} \  - \
a\<f,(- L_{N,\varepsilon,l} f)\>_{\mu^{\Lambda^*_l}_\rho}
\ \leq \ 2 \sup_{z\geq A} |H(z)|\;.
$$
Since $H(n)$ vanishes as $n\uparrow\infty$, the right hand side can be
made arbitrarily small. \vskip .1cm

\noindent {\it Step 8.}  Let now $S^3_{N,l} = A_l \cap R_{N,l}$ where
$A_l = \{\eta: \eta^l \leq A\}$ and $R_{N,l}= \{ (\eta,\zeta): B\leq
\eta^l + \zeta^l \leq 2C_1\log N\}$. This case is the difficult part
of the proof and is treated in Lemma \ref{sl01} below.
\end{proof}

The proof of Lemma \ref{sl01} is reserved to the next subsection.

\begin{lemma}
\label{sl01}
Suppose that $B>4A$. Then, for every $a>0$,
\begin{equation}
\label{fl02}
\begin{split}
& \lim_{l \to\infty}\limsup_{\varepsilon\to 0}\limsup_{N\to\infty} \\
& \quad \sup_f
\Big\{ \int  \mb 1\{R_{N,l} \} \, \mb 1\{A_l \} \, f(\eta, \zeta)^2
\, d \mu^{\Lambda_l^*}_\rho
\; - \; a \< f,  (- L_{N, \varepsilon, l} f) \>_{\mu_{\Lambda_l^*,
    \rho}} \Big\} \;\le\; 0 \;,
\end{split}
\end{equation}
where the supremum over $f$ is over functions $f: \bb
N_0^{\Lambda^*_l}$ such that $\< f, f\>_{\mu^{\Lambda^*_l}_\rho} =1$.
\end{lemma}

\begin{lemma}
\label{canonical_order}
For $s,r\geq 0$, we have $\mu_{\Lambda_s,r}\ll\mu_{\Lambda_s,r+1}$, and
$\nu_{\Lambda_s,r}\ll\nu_{\Lambda_s,r+1}$.
\end{lemma}

\begin{proof}
The first estimate is \cite[Lemma 4.4]{LSV}.  The second bound has a
similar argument: Note $\nu_{\Lambda_s,r}$ is the unique invariant
measure for the Markov chain on $\Sigma^*_{\Lambda_s,r} = \{\eta:
\eta(0)\geq 1, \sum_{|x|\leq s}\eta(x)=r\}$ generated by
$L^{env}_{\Lambda_s}$.

Since $g$ is increasing, we can couple two systems starting from
configurations $\eta^1\in \Sigma^*_{\Lambda_s,r}$ and $\eta^2\in
\Sigma^*_{\Lambda_s,r+1}$ such that $\eta^1\leq \eta^2$
coordinatewise, so that the ordering is preserved at later times.
Hence, in the limit we obtain
$\lim_{t\uparrow\infty}\eta^1_t=\nu_{\Lambda_s,r}$,
$\lim_{t\uparrow\infty}\eta^2_t=\nu_{\Lambda_s,r+1}$, and
$\nu_{\Lambda_s,r}\ll\nu_{\Lambda_s,r+1}$.
\end{proof}

Recall the set $\hat\Lambda_l = \{-l,\ldots,3l+1\}$.
\begin{lemma}
\label{lclt2}
Let $H:\R_+\rightarrow \R_+$ be a positive, bounded, Lipschitz
function which vanishes at infinity.  Then,
we have
\begin{equation*}
\limsup_{l \to \infty} \sup_{k \geq 0} \ \Big| \,
E_{\nu_{\hat\Lambda_{l},k}}\Big [ H(\eta^l(0)) - H(\eta^l(2l+1))
\Big ] \, \Big | \;=\;0\;.
\end{equation*}
\end{lemma}

\begin{proof} The argument is in three parts.
\vskip .1cm

\noindent {\it Step 1.}  Fix $\epsilon>0$ and consider $(k,l)$ such
that $k/|\hat\Lambda_l|\leq \epsilon$.  Add and subtract $H(0)$ in the
absolute value.  Then, the expectation is less than $2\max_{0\leq
  x\leq 2 \epsilon}|H(x)-H(0)| = O(\epsilon)$ given that $H$ is
Lipschitz.  \vskip .1cm

\noindent {\it Step 2.}  Assume now that $\epsilon \leq
k/|\hat\Lambda_l|\leq B_1$. The proof is the same as in Lemma 6.6 in
\cite{jls} for this case.  For the convenience of the reader, we give
it here.  By definition of $\nu_{\hat\Lambda_{l},k}$, the expectation
appearing in the display of the lemma equals
$$
\frac 1{E_{\mu_{\hat\Lambda_{l},k}} [ \eta(0) ]} \,
E_{\mu_{\hat\Lambda_{l},k}}\Big [ \eta(0) \Big\{ H(\eta^l(0)) -
H(\eta^l(2l+1)) \Big\} \Big] \;.
$$
Since the measure is space homogeneous, the denominator is equal to
$\rho_{l,k} = k/|\hat\Lambda_l|$ which is bounded below by $\epsilon$.
In the numerator, $\eta(0)$ can be replaced by $\eta^l(0)$. The
numerator is then
\begin{eqnarray*}
&&E_{\mu_{\hat\Lambda_{l},k}}\Big [ \Big\{ \eta^l(0) - \rho_{l,k}\Big\}
\Big\{ H(\eta^l(0)) - H(\eta^l(2l+1)) \Big\} \Big] \\
&&\ \ \ \ \ \ \ +\;
\rho_{l,k} \, E_{\mu_{\hat\Lambda_{l},k}}\Big [ H(\eta^l(0)) -
H(\eta^l(2l+1)) \Big]\;.
\end{eqnarray*}
The second term vanishes because the measure $\mu_{\hat\Lambda_{l},k}$
is space homogeneous.  The first term, as $H$ is bounded, is
absolutely dominated by $2\|H\|_{L^\infty} E_{\mu_{\hat\Lambda_{l},k}} [ \, |
\eta^l(0) - \rho_{l,k} |\, ]$. By \cite[Appendix II.1 Corollary 1.4]{kl},
this expression is less than or equal to
\begin{equation*}
C_0 E_{\mu_{\rho_{l,k}}^{\hat\Lambda_{l}}} \Big[ \, \big | \eta^l(0)
- \rho_{l,k} \big |\, \Big ] \;\le\; C_0 \, \sigma (\rho_{l,k})
\, l^{-1/2}\;,
\end{equation*}
for some constant $C_0$ where $\sigma (\rho)$ stands for the variance
of $\xi(0)$ under $\mu_\rho$. Since $\epsilon\leq \rho_{l,k}\leq B_1$,
$\sigma(\rho_{l,k})$ is bounded.  Hence, this expression vanishes as
$l\uparrow\infty$.  \vskip .1cm

\noindent {\it Step 3.} Suppose now $k/|\hat\Lambda_l|\geq B_1$. We
shall prove that in this range both expectations are small because
$H(x)$ vanishes as $x\uparrow\infty$.  Fix $A>0$. Introducing the
indicator of the set $\eta^l(0)\leq A$ and replacing the Palm measure
$\nu_{\hat\Lambda_l,k}$ by the homogeneous measure
$\mu_{\hat\Lambda_l,k}$, we get that
\begin{equation*}
\begin{split}
& E_{\nu_{\hat\Lambda_l,k}}[H(\eta^l(0))] \;\leq\;
E_{\nu_{\hat\Lambda_l,k}} \big[ H(\eta^l(0))\, \mb 1\{ \eta^l(0)\leq
A\} \big] \;+\; \sup_{x\geq A} H(x) \\
&\qquad =\; \frac 1{\rho_{l,k}}\,
E_{\mu_{\hat\Lambda_l,k}} \big[ \eta(0)\, H(\eta^l(0))\, \mb 1\{ \eta^l(0)\leq
A\} \big] \;+\; \sup_{x\geq A} H(x)
\end{split}
\end{equation*}
because $E_{\mu_{\hat\Lambda_l,k}}[\eta(0)] = \rho_{l,k}$. In the last
expectation, we may replace $\eta(0)$ by $\eta^l(0)$ which is bounded
by $A$. We may also estimate $H$ by $\|H\|_{L^\infty}$ and bound below
the density $\rho_{l,k}$ by $B_1$. The previous expression is thus
less than or equal to
\begin{equation*}
\frac{A\|H\|_{L^\infty}}{B_1} \;+\; \sup_{x\geq A}H(x)\; ,
\end{equation*}
which can be made arbitrarily small if $A$ is chosen large enough and
then $B_1$.

It remains to prove that the second expectation appearing in the
statement of the lemma is small in this range of
densities. Introducing the indicator of the set $\{ \eta^l(2l+1)\leq
A\}$ we get that
\begin{equation*}
E_{\nu_{\hat\Lambda_l,k}} \big[ H(\eta^l(2l+1)) \big] \;\leq\;
\|H\|_{L^\infty} \, \nu_{\hat\Lambda_l,k} \big\{ \eta^l(2l+1)\leq A \big\}
\;+\; \sup_{x\geq A} H(x)\;.
\end{equation*}
Since the event $\{ \eta^l(2l+1)\leq A \}$ is decreasing and $k\ge B_1
|\hat\Lambda_l|$, by Lemma \ref{lclt2}, we may bound the previous
probability by $\nu_{\hat\Lambda_l,K} \{ \eta^l(2l+1)\leq A \}$, where
$K= B_1 |\hat\Lambda_l|$. At this point, by the same reasons argued
above, we obtain that
\begin{equation*}
E_{\nu_{\hat\Lambda_l,k}}[H(\eta^l(2l+1))] \;\leq\;
\frac{\|H\|_{L^\infty}}{B_1} \,
E_{\mu_{\hat\Lambda_l,K}} \big[ \eta^l(0) \, \mb 1\{\eta^l(2l+1)\leq A\}
\big] \;+\; \sup_{x\geq A} H(x)\;.
\end{equation*}
Note that $\eta^l(0) \le B_1 |\hat\Lambda_l|/(2l+1) = 2B_1$.  Hence,
by \cite[Corollary 1.4, Appendix 2.1]{kl}, the previous expectation is
bounded by
$$
2\, \|H\|_{L^\infty} \mu_{B_1} \{ \eta^l(2l+1)\leq A \} \;+\;
\frac{C_0}{l}
$$
for some finite constant $C_0$.  This expression vanishes as
$l\uparrow\infty$ by the law of large numbers provided $B_1>A$. This
concludes the proof of the lemma.
\end{proof}

\subsection{Proof of Lemma \ref{sl01}}

Fix $B>4A$.  The proof is divided in three steps.  Recall the notation
developed in Step 5 of the proof of Lemma \ref{ec2}.  \vskip .1cm

\noindent {\it Step 1.}  The first integral in \eqref{fl02} can be
rewritten as
\begin{equation*}
\sum_{j,k} \mu_{\rho, l}(j) \, \mu_{\rho, l}(k)
\int \int  f(\eta, \zeta)^2  \, \mu_{\Lambda^-_l, j} (d\eta)
\, \mu_{\Lambda^+_l, k} (d\zeta)  \;,
\end{equation*}
where the sum is performed over all indices $j$, $k$ such that $0\le
j\le A(2l+1)$, $k\ge 0$, $B(2l+1)\le j+k\le \theta_{N,l} :=
2C_1(2l+1)\log N$, $\mu_{\rho, l}(m) = \mu_\rho(\sum_{x\in\Lambda_l}
\eta(x) = m)$ and $\mu_{\Lambda^\pm_l, m}$ is the canonical measure on
the cube $\Lambda^\pm_l$ concentrated on configurations with $m$
particles.

Fix two integers $j,k\geq 0$ such that $B(2l+1) \le j+k\le
\theta_{N,l}$.  We claim that there exists a function $W_N(l)$ such
that $W_N(l)=o(N)$ for fixed $l$ and
\begin{equation}
\label{fl01}
\begin{split}
& \int \int  f(\eta, \zeta)^2  \, \mu_{\Lambda^-_l, j} (d\eta)
\, \mu_{\Lambda^+_l, k} (d\zeta) \; - \;
\Big\{ \int \int  f(\eta, \zeta)  \, \mu_{\Lambda^-_l, j} (d\eta)
\, \mu_{\Lambda^+_l, k} (d\zeta) \Big\}^2 \\
&\qquad  \le\;   W_N(l) \Big\{
D_{\Lambda^-_l}(\mu_{\Lambda^*_l, j, k} , f) \; +\;
D_{\Lambda^+_l}( \mu_{\Lambda^*_l, j, k} ,f)\Big\} \;,
\end{split}
\end{equation}
where $\mu_{\Lambda^*_l, j, k}$ represents the measure
$\mu_{\Lambda^-_l, j} \, \mu_{\Lambda^+_l, k}$ and
$D_{\Lambda^\pm_l}(\mu_{\Lambda^*_l, j, k}, f)$ is the Dirichlet form
defined in \eqref{fl03} with the canonical measure
$\mu_{\Lambda^*_l, j, k}$ in place of the grand canonical measure
$\mu^{\Lambda^*_l}_\rho$.

To prove claim \eqref{fl01}, recall $W(l, j)$ is the inverse of
the spectral gap of the generator of the zero range process in which
$j$ particles move on a cube of length $2l+1$. By definition of
$W(l,j)$, for each configuration $\zeta$,
\begin{equation*}
\begin{split}
& \int  f(\eta, \zeta)^2  \, \mu_{\Lambda^-_l, j} (d\eta) \;-\;
\Big\{ \int  f(\eta, \zeta)  \, \mu_{\Lambda^-_l, j} (d\eta) \Big\}^2 \\
&\quad \;\le\; W(l,j) \sum_{x=-l}^{l -1} \int g(\eta(x))
\{f(\eta^{x,x+1}, \zeta) - f(\eta, \zeta)\}^2  \, \mu_{\Lambda^-_l,
  j} (d\eta)\; .
\end{split}
\end{equation*}
Integrating with respect to $\mu_{\Lambda^+_l, k} (d\zeta)$ we get
that
\begin{eqnarray*}
&& \int \mu_{\Lambda^+_l, k} (d\zeta)
\int  f(\eta, \zeta)^2  \, \mu_{\Lambda^-_l, j} (d\eta)  \\
&&\ \ \ \ \  \;\le\;\quad \int \mu_{\Lambda^+_l, k} (d\zeta) \,
\Big\{ \int  f(\eta, \zeta)  \, \mu_{\Lambda^-_l, j} (d\eta) \Big\}^2
\;+\; W(l,j)  D_{\Lambda^-_l}(\mu_{\Lambda^*_l, j, k} , f)\;.
\end{eqnarray*}

Let
\begin{equation*}
h(\zeta) \;=\; \int  f(\eta, \zeta)  \, \mu_{\Lambda^-_l, j}
(d\eta)\; .
\end{equation*}
By definition of the spectral gap,
\begin{equation*}
\int h(\zeta)^2 \mu_{\Lambda^+_l, k} (d\zeta) \;-\;
\Big\{ \int h(\zeta) \, \mu_{\Lambda^+_l, k} (d\zeta) \Big\}^2
\;\le\; W(l,k) \, D_{\Lambda^+_l}(\mu_{\Lambda^+_l, k} , h)\;.
\end{equation*}
By Schwarz inequality
\begin{equation*}
D_{\Lambda^+_l}(\mu_{\Lambda^+_l, k} , h) \;\le\;
D_{\Lambda^+_l}(\mu_{\Lambda^*_l, j, k} , f)\;.
\end{equation*}
This proves \eqref{fl01}, applying the estimate on the spectral gap in
Lemma \ref{spec_gap} when $g$ satisfies (B), or by assumption when $g$
satisfies (SL).

Multiplying both sides of \eqref{fl01} by $\mu_{\rho, l}(j) \,
\mu_{\rho, l}(k)$ and summing over $j$ and $k$ such that $0\le j\le
A(2l+1)$, $k\ge 0$, $B(2l+1) \le j+k\le \theta_{N,l}$, we see that to
prove the lemma it is enough to show that for every $a>0$
\begin{equation}
\label{fl04}
\sum_{j,k} \mu_{\rho, l}(j) \, \mu_{\rho, l}(k)
\Big\{ \int \int  f(\eta, \zeta)  \, \mu_{\Lambda^-_l, j} (d\eta)
\, \mu_{\Lambda^+_l, k} (d\zeta) \Big\}^2
\; - \; a \, \< f,  (- L_{N, \varepsilon, l} f) \>_{\mu_{\Lambda_l^*,
    \rho}}
\end{equation}
vanishes as $N\uparrow\infty$, $\varepsilon\downarrow 0$,
$l\uparrow\infty$. \vskip .1cm

\noindent {\it Step 2.}
To estimate \eqref{fl04}, let
$$
F(j,k) \;=\; \int \int  f(\eta, \zeta)  \, \mu_{\Lambda^-_l, j} (d\eta)
\, \mu_{\Lambda^+_l, k} (d\zeta)  \;.
$$
We now claim there exists $W_N(l)$, where $W_N(l)=o(N)$ for fixed $l$,
and a finite constant $C_0$ such that
\begin{equation}
\label{fl05}
\begin{split}
& \sum_{j,k} \mu_{\rho, l}(j) \, \mu_{\rho, l}(k) \,
[F(j+1,k-1) - F(j,k)]^2 \\
&\quad
\;\le\;  W_N(l)\,
\Big\{ D_{\Lambda^-_l}(\mu^{\Lambda^*_l}_\rho ,
f) \; +\; D_{\Lambda^+_l}(\mu^{\Lambda^*_l}_\rho , f) \Big\}
\;+\; C_0 \, l^5\, D_{0}(\mu^{\Lambda^*_l}_\rho , f)\;,
\end{split}
\end{equation}
where the sum is over all $j$ and $k$ such that $j\ge 0$, $k\ge 1$,
$B(2l+1)\le j+k\le \theta_{N,l}$.

To prove \eqref{fl05}, note that since $\mu_{\Lambda^-_l, j}
(d\eta)$ is the canonical measure,
\begin{equation*}
F(j+1,k-1) \;=\; \int \mu_{\Lambda^+_l, k-1} (d\zeta)
\int  f(\eta, \zeta)  \frac 1{j+1} \sum_{x\in \Lambda^-_l}
\eta(x) \, \mu_{\Lambda^-_l, j+1} (d\eta) \;.
\end{equation*}
Changing variables $\eta' = \eta - \mf d_x$, the previous expression
becomes
\begin{equation*}
\frac 1{2l+1} \sum_{x\in \Lambda^-_l} \int \mu_{\Lambda^+_l, k-1}
 (d\zeta) \int f(\eta + \mf d_x, \zeta) \, h_{l, j}(\eta(x))
\, \mu_{\Lambda^-_l, j} (d\eta)\;,
\end{equation*}
where
\begin{equation*}
h_{l, j}(\eta(x)) \;=\; \frac {2l+1}{j+1}\frac{\varphi(\rho) \mu_{\rho, l}(j)}
{\mu_{\rho, l}(j+1)}
\frac{1 + \eta(x)}{g(1 + \eta(x))}\;\cdot
\end{equation*}
Note that $h_{l, j}(\eta(x))$ has mean equal to $1$ with respect to
$\mu_{\Lambda^-_l, j} (d\eta)$.

Changing variables $\zeta' = \zeta + \mf d_y$, the previous integral
becomes
\begin{equation*}
\frac 1{(2l+1)^2} \sum_{\substack{x\in \Lambda^-_l \\ y\in \Lambda^+_l}}
\int \mu_{\Lambda^+_l, k} (d\zeta) \int f(\eta + \mf d_x, \zeta -\mf d_y)
\, h_{l, j}(\eta(x)) \, e_{l, k}(\zeta(y))
\, \mu_{\Lambda^-_l, j} (d\eta)\;,
\end{equation*}
where
\begin{equation*}
e_{l, k}(\zeta(y)) \;=\; \frac {\mu_{\rho, l}(k)}
{\mu_{\rho, l}(k-1)} \frac{g(\zeta(y))}{\varphi(\rho)}
\end{equation*}
has mean $1$ with respect to $\mu_{\Lambda^+_l, k}$.

This identity permits to write $F(j+1,k-1) - F(j,k)$
as the sum of two terms
\begin{eqnarray}
\label{grad_F_eqn}&&   \
\frac 1{(2l+1)^2} \sum_{\substack{x\in \Lambda^-_l \\ y\in
    \Lambda^+_l}} \int \mu_{\Lambda^+_l, k}
(d\zeta) \int \{ f(\eta + \mf d_x, \zeta -\mf d_y) - f(\eta, \zeta)\}
\nonumber\\
&&\ \ \ \ \ \ \ \ \ \ \ \ \ \ \ \ \ \ \ \ \ \
\ \ \ \ \ \ \ \ \ \ \ \ \ \ \ \
\cdot \, h_{l, j}(\eta(x)) \, e_{l, k}(\zeta(y))
\, \mu_{\Lambda^-_l, j} (d\eta) \\
&& +\;
\int \mu_{\Lambda^+_l, k} (d\zeta) \int f(\eta, \zeta)
\, \frac 1{(2l+1)^2} \sum_{\substack{x\in \Lambda^-_l \\ y\in
    \Lambda^+_l}} [ h_{l, j}(\eta(x))e_{l, k}(\zeta(y)) - 1]
\, \mu_{\Lambda^-_l, j} (d\eta) \;.
\nonumber
\end{eqnarray}
Since $(a+b)^2 \le 2a^2 + 2b^2$, $[F(j+1,k-1) - F(j,k)]^2$ is bounded
above by the sum of two terms. One term, corresponding to the last
line of (\ref{grad_F_eqn}), is equal to
\begin{equation}
\label{fl06}
2 \, \Big( \Big \< f \,;\, \frac 1{(2l+1)^2}
\sum_{\substack{x\in \Lambda^-_l \\ y\in
    \Lambda^+_l}} h_{l, j}(\eta(x)) \, e_{l, k}(\zeta(y)) \Big>_{l,j,k}
\Big)^2 \;,
\end{equation}
where $\< F ; G\>_{l,j,k}$ denotes the covariance of $F$ and $G$ with
respect to $\mu_{\Lambda^+_l, k} \, \mu_{\Lambda^-_l, j}$.

Since
\begin{equation}
\label{l5}
\frac{\varphi(\rho) \mu_{\rho,l}(r)}{\mu_{\rho,l}(r+1)}
\ = \ E_{\mu_{\Lambda^+_l,r+1}}[g(\eta(1))]\; ,
\end{equation}
we have that
\begin{equation*}
\begin{split}
h_{l, j}(\eta(x)) \, e_{l, k}(\zeta(y)) \; & =\;
\frac{2l+1}{j+1}
\frac{\varphi(\rho)\mu_{\rho,l}(j)}{\mu_{\rho,l}(j+1)}
 \frac{1+\eta(x)}{g(1+\eta(x))}\frac{\mu_{\rho,l}(k)g(\zeta(y))}
{\mu_{\rho,l}(k-1)\varphi(\rho)}\\
& =\; \frac{2l+1}{j+1}
\frac{1+\eta(x)}{g(1+\eta(x))} \, g(\zeta(y))\,
\frac{E_{\mu_{\Lambda_l^-,j+1}}[g(\eta(-1))]}
{E_{\mu_{\Lambda_l^+,k}}[g(\zeta(1))]} \;\cdot
\end{split}
\end{equation*}

We claim that under the measure $\mu_{\Lambda^+_l, k} \,
\mu_{\Lambda^-_l, j}$,
\begin{equation}
\label{l3b}
h_{l, j}(\eta(x)) \, e_{l, k}(\zeta(y)) \; \le \;
C_0  \, l\, \frac{g(\zeta(y))}
{E_{\mu_{\Lambda_l^+,k}}[g(\zeta(1))]}
\end{equation}
for some finite constant $C_0$ depending only on $a_0$, $a_1$.  This
bound is simple to derive when when $g$ fulfills assumption (B).  On
the other hand, under the assumptions (SL), since $g$ is increasing,
$E_{\mu_{\Lambda_l^-,j+1}} [g(\eta(-1))] \le g(j+1)$, and since
$g(k)/k$ is decreasing, under the measure $\mu_{\Lambda^-_l, j}$,
$[1+\eta(x)]/g(1+\eta(x))$ is less than or equal to $(j+1)/g(j+1)$.
This proves \eqref{l3b}. This is the only place where we use that
$g(k)/k$ is decreasing in $k$ in the condition (SL).

Therefore, by Schwarz inequality, (\ref{fl06}) is bounded above by
\begin{equation*}
C_0 \, l \, \< f \,;\, f \big>_{l,j,k} \,
\sum_{y\in \Lambda^+_l} \frac{E_{\mu_{\Lambda_l^+,k}}[g(\zeta(y))^2]}
{E_{\mu_{\Lambda_l^+,k}}[g(\zeta(1))]^2}
\end{equation*}
for some finite constant $C_0$. In view of \eqref{l5}, the fact that
$g(m+1)-g(m) \le a_2$, which follows from assumption (B) or from
assumption (SL2), and Lemma \ref{canonical_order},
\begin{equation*}
\begin{split}
& E_{\mu_{\Lambda_l^+,k}}[g(\zeta(y))^2] \;=\;
E_{\mu_{\Lambda_l^+,k}}[g(\zeta(y))]\,
E_{\mu_{\Lambda_l^+,k-1}}[g(\zeta(y)+1)] \\
& \qquad  \;\le\;
E_{\mu_{\Lambda_l^+,k}}[g(\zeta(y))]\, \Big\{
a_2 + E_{\mu_{\Lambda_l^+,k-1}}[g(\zeta(y))] \Big\} \\
& \qquad  \;\le\;
E_{\mu_{\Lambda_l^+,k}}[g(\zeta(y))]\, \Big\{
a_2 + E_{\mu_{\Lambda_l^+,k}}[g(\zeta(y))] \Big\}\;.
\end{split}
\end{equation*}

As $g$ is increasing, by Lemma \ref{canonical_order},
$E_{\mu_{\Lambda^+_l,1}}[g(\zeta(1))]\leq
E_{\mu_{\Lambda^+_l,k}}[g(\zeta(1))]$.  Hence, since $a_0\mb 1\{r\geq
1\}\leq g(r)$, and since $E_{\mu_{\Lambda^+_l,1}}[\mb 1\{\zeta(1)\geq
1\}] =E_{\mu_{\Lambda^+_l,1}}[\zeta(1)]= (2l+1)^{-1}$, we have that
\begin{equation}
\label{g_canonical}
\frac{a_0}{2l+1} \;=\;  a_0E_{\mu_{\Lambda^+_l,1}}[\mb 1\{\zeta(1)\geq
1\}] \;\le\; E_{\mu_{\Lambda^+_l,k}}[g(\zeta(1))]\;.
\end{equation}

It follows from this estimate and from the previous bound that
(\ref{fl06}) is less than or equal to
\begin{equation*}
C_0 \, l^3 \, \< f \,;\, f \big>_{l,j,k} \;.
\end{equation*}

Multiply this expression by $\mu_{\rho, l}(j) \, \mu_{\rho, l}(k)$,
recall the bound \eqref{fl01}, and sum over $j$ and $k$ such that
$j\ge 0$, $k\ge 1$, $B(2l+1)\le j+k\le \theta_{N,l}$, to get that
\eqref{fl06} is bounded by
\begin{equation*}
W_N(l) \Big\{
D_{\Lambda^-_l}(\mu^{\Lambda^*_l}_\rho , f) \; +\;
D_{\Lambda^+_l}( \mu^{\Lambda^*_l}_\rho ,f)\Big\} \;,
\end{equation*}
where $W_N(l)=o(N)$ for fixed $l$.

We now estimate the first term in the decomposition
(\ref{grad_F_eqn}). By Schwarz inequality and by the bounds
\eqref{l3b}, \eqref{g_canonical}, the square of this expression is less
than or equal to
\begin{equation*}
C_0\, l^3 \sum_{\substack{x\in \Lambda^-_l \\ y\in
    \Lambda^+_l}} \int \mu_{\Lambda^+_l, k}
(d\zeta) \int g(\zeta(y))
\{ f(\eta + \mf d_x, \zeta -\mf d_y) - f(\eta, \zeta)\}^2
\, \mu_{\Lambda^-_l, j} (d\eta)
\end{equation*}
for some finite constant $C_0$.  The sum over $j\geq 0$, $k\geq 1$ of
this expression, when multiplied by $\mu_{\rho,l}(k)\mu_{\rho,l}(j)$,
is bounded by
\begin{equation*}
C_0\, l^3 \sum_{\substack{x\in \Lambda^-_l \\ y\in
    \Lambda^+_l}} \int \mu_{\Lambda^+_l, \rho}
(d\zeta) \int g(\zeta(y))
\{ f(\eta + \mf d_x, \zeta -\mf d_y) - f(\eta, \zeta)\}^2
\, \mu_{\Lambda^-_l, \rho} (d\eta).
\end{equation*}
Changing variables $\zeta' = \zeta - \mf d_y$, adding and subtracting
in the expression inside braces the terms $f(\eta + \mf d_{-1},
\zeta)$, $f(\eta, \zeta + \mf d_1)$, we estimate the previous
expression by
\begin{equation*}
C_0\, l^5 \, D_{0}(\mu_{\Lambda^*_l, \rho} , f)
\;+\;  C_0 \, l^5 \, \Big\{
D_{\Lambda^-_l}(\mu_{\Lambda^*_l, \rho} , f) \; +\;
D_{\Lambda^+_l}( \mu_{\Lambda^*_l, \rho} ,f)\Big\}
\end{equation*}
for some constant $C_0$. This proves claim \eqref{fl05}.  \vskip .1cm

\noindent {\it Step 3.} In view of \eqref{fl04} and of \eqref{fl05},
to prove the lemma it is enough to show that for every $a>0$
\begin{eqnarray*}
&&\lim_{l\to\infty}\limsup_{\varepsilon\to 0}  \limsup_{N\to\infty} \sup_F
\Big\{ \sum_{j,k}
F(j,k)^2 \mu_{\rho, l}(j) \, \mu_{\rho, l}(k) \\
&&\qquad\qquad\qquad\quad
-\; a \, \varepsilon^{-1} \, \sum_{j,k} [F(j+1,k-1) - F(j,k)]^2 \,
\mu_{\rho, l}(j) \, \mu_{\rho, l}(k) \Big\} \;=\; 0\;,
\end{eqnarray*}
where the first sum is carried over all $0\le j\le A(2l+1)$, $k\ge 0$,
$B(2l+1)\le j+k\le \theta_{N,l}$, the second sum is carried over all
$j\ge 0$, $k\ge 1$, $B(2l+1)\le j+k\le \theta_{N,l}$ and where the
supremum is carried over all functions $F: \bb N_0 \times \bb N_0 \to
\bb R$ such that $\sum_{j,k \ge 0} F(j,k)^2 \mu_{\rho, l}(j) \,
\mu_{\rho, l}(k) = 1$.

The expression inside braces can be bounded by
\begin{equation}
\label{RW_eq}
\begin{split}
& \sum_{M=B(2l+1)}^{\theta_{N,l}} \mu_{\rho, \Lambda^*_l}(M)\, Z_{M}(F)
\, \Big\{ \sum_{j=0}^{A(2l+1)} G(j)^2\, \mu_{\Lambda^*_l, M}(j) \\
&\qquad\qquad\qquad\qquad\qquad
\; -\; a \, \varepsilon^{-1} \sum_{j=0}^{B(2l+1)-1} [G(j+1) - G(j)]^2 \,
\mu_{\Lambda^*_l, M}(j)\Big\}\;,
\end{split}
\end{equation}
where
\begin{equation*}
\begin{split}
& \mu_{\rho, \Lambda^*_l}(M) \;=\; \sum_{j=0}^{B(2l+1)}
\mu_{\rho, l}(j) \, \mu_{\rho, l}(M-j) \;,
\quad  \mu_{\Lambda^*_l, M}(j) \;=\; \frac{\mu_{\rho, l}(j)
\, \mu_{\rho, l}(M-j)} {\mu_{\rho, \Lambda^*_l}(M)}\;, \\
& Z_{M}(F) \;=\; \sum_{j=0}^{B(2l+1)} F(j,M-j)^2 \mu_{\Lambda^*_l,
  M}(j)\; , \quad Z_{M}(F) \, G(j)^2 \;=\; F(j,M-j)^2\;.
\end{split}
\end{equation*}
Note that we omitted the dependence on $B$ of the variables
$\mu_{\rho, \Lambda^*_l}(M)$, $\mu_{\Lambda^*_l, M}(j)$, $Z_{M}(F)$
and that $\sum_{0\le j\le B(2l+1)} G(j)^2 \mu_{\Lambda^*_l, M}(j)=1$.

The expression inside braces in (\ref{RW_eq}) can be interpreted in
terms of a random walk on an interval of length $B(2l+1)$ where the
total number of particles $M$ becomes a parameter. In fact, the second
term in braces corresponds to the Dirichlet form of a random walk on
$\{0, \cdots, B(2l+1)\}$ which jumps from $j$ to $j+1$, $0\le j\le
B(2l+1)-1$, at rate $1$ and from $j+1$ to $j$ at rate $r_M
(j+1,j)=\mu_{\Lambda^*_l, M}(j)/\mu_{\Lambda^*_l, M}(j+1)$. By
\eqref{l5},
\begin{equation*}
r_M(j+1,j)\;=\; \frac{\mu_{\rho,l}(j)}{\mu_{\rho,l}(j+1)}
\frac{\mu_{\rho,l}(M-j)}{\mu_{\rho,l}(M-j-1)} \;=\;
\frac{E_{\mu_{\Lambda^+_l,j+1}}[g(\eta(1))]}
{E_{\mu_{\Lambda^+_l,M-j}}[g(\eta(1))]} \; \cdot
\end{equation*}

We claim that this random walk has a spectral gap
\begin{equation}
\label{l7}
\text{$\hat \lambda_{l,B}$
which depends on $B$ and $l$ but is uniform over $M$.}
\end{equation}

Assume first that $g$ satisfies (B). In this case, by
(\ref{g_canonical}), the previous ratio is bounded above by
$a_1a_0^{-1}(2l+1)$ and below by $a_0a_1^{-1}(2l+1)^{-1}$. The jump
rates are therefore bounded below and above by finite constants
independent of $M$, and claim \eqref{l7} follows easily.

Assume now that $g$ satisfies (SL). We claim that for $l$ large
enough,
\begin{equation}
\label{l6}
\lim_{M\to \infty}
\max_{0\le j\le B(2l+1)-1} r_M(j+1,j) \;=\;0\; .
\end{equation}
Indeed, by Lemma \ref{canonical_order}, $E_{\mu_{\Lambda^+_l,j+1}}
[g(\eta(1))]\leq g(B(2l+1))$. On the other hand, for every $D\leq
D'(2l+1)\le M- B(2l+1) \leq M-j$
\begin{equation*}
\begin{split}
& E_{\mu_{\Lambda^+_l,M-j}}[g(\eta(1))] \;\geq\;
E_{\mu_{\Lambda^+_l,M-j}} \big[ g(\eta(1)) \, \mb 1\{\eta(1)\geq D\}
\big] \\
&\quad \geq\; g(D)\, \mu_{\Lambda^+_l,D'(2l+1)} \{ \eta(1)\geq D \}
\;\geq\; g(D)\, \big( \mu_{D'}\{\eta(1)\geq D\} - C_0/l \big)\; ,
\end{split}
\end{equation*}
where the last inequality follows from the equivalence of ensembles
\cite[Appendix 2.1, Corollary 1.7]{kl} and $C_0$ is a finite constant.
The right-side can be made arbitrarily large since
$\lim_{D'\uparrow\infty} \mu_{D'} \{ \eta(1)\geq D\}=1$ and
$\lim_{D\uparrow\infty} g(D)=\infty$. This proves claim \eqref{l6}.

Fix $M_0$ and $l$ large enough so that $r_M(j+1,j)<1$ for all $0\leq
j\leq B(2l+1)-1=Q$, $M\geq M_0$.  The stationary probabilities for
the corresponding birth-death chain on the interval can be expressed
as
\begin{equation*}
\pi_j \ = \ \frac{\left(\prod_{s=0}^{j-1}r_M(s+1,s)\right)^{-1}}
{1+ \sum_{t=1}^Q \left(\prod_{s=0}^{t-1}r_M(s+1,s)\right)^{-1}}\;,
\end{equation*}
where the empty product in the numerator is taken to be $1$ when
$j=0$.  We note by construction that $\pi_j\leq \pi_{j+1}$. We have
the Poincar\'e inequality:
\begin{equation*}
\begin{split}
& {\rm Var}_\pi(f)
\;\leq\; \sum_{x,y} \pi_x\pi_y (f(x)-f(y))^2
\; \leq\;  2Q\sum_{y>x}\pi_x\pi_y \sum_{z=x}^{y-1} (f(z)-f(z+1))^2\\
&\quad \leq\; 2Q\sum_{y>x}\pi_y\sum_{z=x}^{y-1}\pi_z(f(z)-f(z+1))^2
\;\leq\; 2Q^2\sum_{z=0}^{Q-1}\pi_z (f(z)-f(z+1))^2  \;.
\end{split}
\end{equation*}
Hence, for large $M$, the inverse of the spectral gap, $\hat
\lambda_{l,B}^{-1}$, is bounded by $C_0l^2$ for some constant $C_0$
depending only on $B$.  For $M\leq M_0$, recalling \eqref{l3b}, we may
obtain a lower and an upper bound on $r_M(j+1,j)$ which depend only on
$a_0$, $a_1$, $a_2$, $B$, $l$ and $M_0$. It is easy to show that in
this case the inverse gap, $\hat \lambda_{l,B}^{-1}$, is bounded by a
constant which depends only on $B$, $l$ and $M_0$. This concludes the
proof of claim \eqref{l7}. \smallskip

At this point, we may apply the Rayleigh bound \cite[Theorem
A3.1.1]{kl} to estimate the expression in braces in (\ref{RW_eq}). Let
$V_0 = \mb 1\{0, \dots, A(2l+1)\}$ and let $\bar V_0 = V_0 -
E_{\mu_{\Lambda^*_l,M}}[V_0]$ so that
$$
\sum_{j=0}^{A(2l+1)} G(j)^2\, \mu_{\Lambda^*_l, M}(j) \ =\
\sum_{j=0}^{B(2l+1)} V_0(j)\, G(j)^2\, \mu_{\Lambda^*_l, M}(j) \;.
$$
Since $\|V_0\|_{L^\infty}\leq 1$, by the Rayleigh expansion and by the
spectral gap, the display in braces in (\ref{RW_eq}) is bounded by
\begin{eqnarray*}
\sum_{j=0}^{A(2l+1)} \mu_{\Lambda^*_l, M}(j)
\;+\; \frac{a^{-1} \, \varepsilon\, \hat \lambda_{l,B}^{-1}}
{1-2 \, a^{-1} \,\varepsilon \,\hat \lambda_{l,B}^{-1}} \; \cdot
\end{eqnarray*}
The second term vanishes as $\varepsilon \downarrow 0$. To bound the
first term, let $\alpha_j = \mu_{\rho,l}(j) \mu_{\rho,l}(M-j)$, $0\leq
j\leq M$. Since $B>2A$, the first term in the last formula is equal to
\begin{equation*}
\frac{\sum_{j=0}^{A(2l+1)} \alpha_j}{\sum_{j=0}^{B(2l+1)}\alpha_j}
\;\leq\; \frac{\sum_{j=0}^{A(2l+1)} \alpha_j}
{\sum_{j=A(2l+1)}^{2A(2l+1)}\alpha_j} \ \leq \
\max_{0\leq j\leq A(2l+1)} \frac{\alpha_j}{\alpha_{j+A(2l+1)}}
\;\cdot
\end{equation*}

As above in calculating $r_M(j+1,j)$, since $g$ is increasing and
$M\geq B(2l+1)$, if $R= A(2l+1)$, , $S=(B-2A)(2l+1)$, by Lemma
\ref{canonical_order},
\begin{equation*}
\frac{\alpha_j}{\alpha_{j+R}} \;=\;
\prod_{k=j}^{j+R-1} \frac{E_{\mu_{\Lambda^+_l,k+1}}[g(\eta(1))]}
{E_{\mu_{\Lambda^+_l,M-k}}[g(\eta(1))]} \;\le\;
\bigg \{ \frac{E_{\mu_{\Lambda^+_l,2R}}[g(\eta(1))]}
{E_{\mu_{\Lambda^+_l,S}}[g(\eta(1))]} \bigg\}^{R}\;.
\end{equation*}
Since $B>4A$, by \cite[Corollary 1.6; Appendix 2.1]{kl}, for all large
$l$, we have
\begin{equation*}
E_{\mu_{\Lambda^+_l,2R}}[g(\eta(1))] \; \leq \;
\varphi(2A) + \frac{C_0}{l} \;<\;
\varphi(B-2A) - \frac{C_0}{l} \ \leq \
E_{\mu_{\Lambda^+_l,S}}[g(\eta(1))]
\end{equation*}
for some finite constant $C_0$.  Hence, the expression appearing in
the previous displayed formula vanishes exponentially fast as
$l\uparrow\infty$.  This concludes the proof of the lemma.  \qed

\subsection{Proof of Theorem \ref{replacement1}}
\label{replacement1_section}

Given Lemmas \ref{l3} and \ref{l4}, the argument is similar to that in
\cite{jls}.  Recall $H(\rho) = E_{\nu_\rho}[h]$, $H_l(\eta) =
H(\eta^l(0))$, and $\bar{H_l}(\rho) = E_{\mu_\rho}[H_l]$.  Then, we
have that
\begin{eqnarray*}
&&\bb E^N \Big[ \, \Big|\int_0^t \big\{h(\eta_s) -
\frac{1}{\epsilon N}\sum_{x=1}^{\epsilon
  N}\bar{H_l}(\eta_s^{\varepsilon N}(x)) \big\} ds \Big|\, \Big] \\
&& \ \ \ \ \ \ \ \ \ \leq \; \bb E^N \Big[ \, \Big|\int_0^t \big\{h(\eta_s)
-H(\eta_s^l(0))\big\}ds \Big|\, \Big] \\
&&\ \ \ \ \ \ \ \ \ \ \ + \; \bb E^N \Big[ \,
\Big|\int_0^t \Big\{H(\eta_s^l(0))
-\frac{1}{\epsilon N} \sum_{x=1}^{\epsilon N}
H(\eta_s^l(x))\Big\}ds \Big|\, \Big] \\
&&\ \ \ \ \ \ \ \ \ \ \ + \; \bb E^N \Big[ \, \Big|\int_0^t
\Big\{\frac{1}{\epsilon N}
\sum_{x=1}^{\epsilon N} \Big(H(\eta_s^l(x))-\bar{H_l}
(\eta_s^{\varepsilon N}(x)\Big) \Big\}ds \Big|\, \Big]\;.
\end{eqnarray*}

As $h$ and $H$ are bounded, Lipschitz by Lemma \ref{Lip-lemma}, the
first and second terms vanish by Lemmas \ref{l3} and \ref{l4}. The
third term is recast as
\begin{equation*}
\bb E^N \Big[ \, \Big|\int_0^t \Big\{\frac{1}{N} \sum_{x\in {\bb T}_N}
\iota_\epsilon(x/N)\Big(\tau_x H_l(\eta_s)-\bar{H_l}(\eta_s^{\varepsilon
  N}(x))\Big) \Big\}ds \Big|\,  \Big]
\end{equation*}
where $\iota_\epsilon(\cdot) = \epsilon^{-1}1\{(0,\epsilon]\}$. It
vanishes by Proposition \ref{p1} as $N\uparrow \infty$, and
$\varepsilon\downarrow 0$. \qed \smallskip

\begin{lemma}
\label{Lip-lemma}
Let $h:\N\rightarrow \R_+$ be a nonnegative, Lipschitz function for
which there is a constant $C$ such that $kh(k)\leq Cg(k)$ for $k\geq
1$.  Then, $H(\rho)= E_{\nu_\rho}[h(\eta(0))]$ is also nonnegative,
bounded and Lipschitz, and vanishes at infinity.
\end{lemma}

\begin{proof}
It follows from the assumptions of the lemma that $h$ is bounded, as
$g(k)\le a_1k$, and that $h$ vanishes at infinity. Hence, $H$, which
is clearly positive, is also bounded.  We claim that $H$ vanishes at
infinity since
$$
H(\rho) \ \leq \ \sup_{x\geq A}h(x) \;+\;
\frac{1}{\rho}E_{\mu_\rho} \Big[ \eta(0)\, h(\eta(0)) \, \mb
1\{\eta(0)\leq A\} \Big] \ \leq \
\sup_{x\geq A}h(x) \;+\;  \frac{A \|h\|_{L^\infty}}{\rho}\;\cdot
$$

To show $H$ is Lipschitz, it is enough to show $H'$ is absolutely
bounded.  Compute
\begin{equation*}
H'(\rho) \ = \ \frac{\varphi'(\rho)}{\rho\, \varphi(\rho)}
\, \big\< h(\eta(0)) \, \eta(0)^2 \big\>_{\mu_\rho} \;-\;
\Big\{ \frac{1}{\rho^2} + \frac{\varphi'(\rho)}{\varphi(\rho)}\Big\}
\big\<h(\eta(0)) \, \eta(0) \big\>_{\mu_\rho}\;.
\end{equation*}

We first examine this expression for $\rho$ large.  The second term,
by the assumption $kh(k)\leq Cg(k)$, is bounded by
$C\{\varphi(\rho)/\rho^2 + \varphi'(\rho)\}$. A coupling argument
shows that $\varphi'(\rho) \le a$ if $a$ is a Lipschitz constant of
the function $g$. On the other hand, $\varphi(\rho)/\rho^2 \le
a_1/\rho$ because $g(k)\le a_1k$.

Since $kh(k)\leq Cg(k)$ and since $E_{\mu_\rho} [g(\eta(0)) \,
\eta(0)] = \varphi(\rho) \, (1+\rho)$, the first term is bounded by
$C a (1+\rho)/\rho$ if $a$ is a Lipschitz constant for $g$. This
proves that $H'$ is absolutely bounded for $\rho$ large.

It is also not difficult to see that $H'(\rho)$ is bounded for $\rho$
close to $0$.
\end{proof}

\end{document}